\documentclass[11pt]{amsart}
\usepackage{amsfonts,amssymb,amsthm}
\usepackage{amsmath,amscd}
\usepackage{pstricks}
\usepackage{pstricks,pst-node}
\usepackage{mathrsfs}
\usepackage[all]{xy}

\DeclareMathAlphabet{\mathpzc}{OT1}{pzc}{m}{it}

\renewcommand{\subsection}[1]{\vspace{.18in}
\par\noindent\addtocounter{subsection}{1}
\setcounter{equation}{0}{\bf\thesubsection.\hspace{5pt}#1}}

\theoremstyle{definition}
\newtheorem{Rem}[subsection]{Remark}

\theoremstyle{plain}

\newtheorem{Prop}[subsection]{Proposition}
\newtheorem{Thm}[subsection]{Theorem}
\newtheorem{Not}[subsection]{Notation}
\newtheorem{Lem}[subsection]{Lemma}
\newtheorem{Coro}[subsection]{Corollary}
\numberwithin{equation}{subsection}

\newcommand{\afVn}{\sV_\vtg(n)}

\newcommand{\afmsD}{{\mathscr D}^\vtg}

\newcommand{\bssg}{\boldsymbol{\sg}}

\newcommand{\bse}{\boldsymbol{e}}

\newcommand{\Hall}{{{\mathfrak H}_\vtg(n)}}

\newcommand{\bfb}{{\mathbf{b}}}
\newcommand{\bfc}{{\mathbf{c}}}
\newcommand{\bfd}{{\mathbf{d}}}

\newcommand{\bfj}{{\mathbf{j}}}

\newcommand{\bfl}{{\mathbf{0}}}

\def\fS{{\frak S}}

\def\fC{{\frak C}}
\def\fB{{\frak B}}
\def\fM{{\frak M}}
\def\fG{{\frak G}}
\def\fP{{\frak P}}

\newcommand{\fkg}{{\frak g}}

\newcommand{\ms}{\mathscr}

\newcommand{\msD}{\mathscr D}

\newcommand{\msKnh}{\mathscr K_\vtg(n)_h}
\newcommand{\msKnhq}{\mathscr K'_\vtg(n)_h}

\newcommand{\afKn}{\mathsf K_\vtg(n)_\sZ}
\newcommand{\afKnmbz}{\mathsf K_\vtg(n)_\mbz}
\newcommand{\bfKn}{\boldsymbol{{\mathsf K}}(n)}
\newcommand{\hbfKn}{\boldsymbol{\h{\mathsf K}}(n)}

\newcommand{\afhKnk}{\h{{\mathsf K}}_\vtg(n)_\field}

\newcommand{\afKnk}{\mathsf K_\vtg(n)_\field}
\newcommand{\field}{\mathpzc k}

\newcommand{\mpX}{\mathpzc X}
\newcommand{\mpY}{\mathpzc Y}
\newcommand{\mpZ}{\mathpzc Z}

\def\sK{{\mathcal K}}
\def\sL{{\mathcal L}}

\def\sU{{\mathcal U}}
\def\sV{{\mathcal V}}

\def\sX{{\mathcal X}}

\def\sZ{{\mathcal Z}}

\newcommand{\afmbnnh}{\mathbb N_{\vtg,p^h}^{n}}
\newcommand{\afmbznh}{\mathbb Z_{\vtg,p^h}^{n}}
\newcommand{\mbn}{\mathbb N}
\newcommand{\mbq}{\mathbb Q}

\newcommand{\mbz}{\mathbb Z}
\newcommand{\mbf}{\mathbb F}

\newcommand{\ttt}{\mathtt{t}}
\newcommand{\ttu}{\mathtt{u}}
\newcommand{\tts}{\mathtt{s}}

\newcommand{\End}{\operatorname{End}}

\newcommand{\spann}{\operatorname{span}}
\newcommand{\diag}{\operatorname{diag}}

\def\ro{\text{\rm ro}}
\def\co{\text{\rm co}}

\newcommand{\ul}{\underline}

\newcommand{\la}{{\lambda}}
\newcommand{\La}{\Lambda}
\newcommand{\ga}{{\gamma}}

\newcommand{\Th}{\Theta}
\newcommand{\dt}{\delta}

\newcommand{\up}{v}
\newcommand{\vi}{\varphi}
\newcommand{\vih}{\varphi_h}
\newcommand{\psih}{\psi_h}

\newcommand{\al}{\alpha}
\newcommand{\bt}{\beta}
\newcommand{\sg}{\sigma}

\def\th{\theta} 

\newcommand{\ol}{\overline}

\newcommand{\lan}{\langle}
\newcommand{\ran}{\rangle}

\newcommand{\lebr}{\left(}
\newcommand{\ribr}{\right)}

\newcommand{\bblr}{\big(}
\newcommand{\bbrr}{\big)}

\def\br#1{\{ #1 \}}
\def\bpa#1#2{\left({#1\atop #2}\right)}

\def\ggp#1#2{\left[\kern-3.2pt\left[{#1\atop #2}\right]\kern-3.2pt\right]}

\def\leq{\leqslant}\def\geq{\geqslant}
\def\le{\leqslant}\def\ge{\geqslant}

\newcommand{\bop}{\bigoplus}

\newcommand{\ot}{\otimes}

\newcommand{\bin}{\bigcup}

\newcommand{\han}{\subseteq}

\newcommand{\h}{\widehat}
\newcommand{\ti}{\widetilde}

\newcommand{\tiThn}{\ti\Th(n)}

\newcommand{\aftiThnh}{\ti\Th_\vtg(n)_h}

\newcommand{\aftiThnhq}{\ti\Th'_\vtg(n)_h}
\newcommand{\Thnpm}{\Th^\pm(n)}

\newcommand{\afThnpmh}{\Th_\vtg^\pm(n)_h}
\newcommand{\afThnph}{\Th_\vtg^+(n)_h}
\newcommand{\afThnmh}{\Th_\vtg^-(n)_h}

\newcommand{\eap}{{\tt e}^{(A^+)}}
\newcommand{\Eap}{E^{(A^+)}}
\newcommand{\faf}{{\tt f}^{(A^-)}}
\newcommand{\Faf}{F^{(A^-)}}

\newcommand{\Ea}{E^{(A)}}

\newcommand{\tA}{{}^t\!A}

\newcommand{\lra}{\longrightarrow}
\newcommand{\ra}{\rightarrow}
\newcommand{\tra}{\twoheadrightarrow}

\newcommand{\xrk}{\xi_{r,\field}}

\newcommand{\mnmod}{\!\!\!\mod\!}

\newcommand{\snkh}{\tts_\vtg(n)_h}
\newcommand{\snkhr}{\tts_\vtg(n,r)_h}
\newcommand{\Unkhr}{\ttu_\vtg(n,r)_h}

\newcommand{\Un}{\sU_\mbz(\afgl)}
\newcommand{\Unk}{\sU_\field(\afgl)}

\newcommand{\Unkh}{{\ttu}_\vtg(n)_h}
\newcommand{\Unkhp}{{\ttu}^+_\vtg(n)_h}
\newcommand{\Unkhm}{{\ttu}^-_\vtg(n)_h}
\newcommand{\Unkhz}{{\ttu}_\vtg^0(n)_h}

\newcommand{\Unko}{{\ttu}_\vtg(n)_1}
\newcommand{\Unkt}{{\ttu}_\vtg(n)_2}

\newcommand{\vtg}{{\!\vartriangle\!}}

\newcommand{\afuglq}{\sU(\widehat{\frak{gl}}_n)}
\newcommand{\afuglk}{\sU_\field(\widehat{\frak{gl}}_n)}
\newcommand{\afuglqz}{\sU^0(\widehat{\frak{gl}}_n)}
\newcommand{\afuglqp}{\sU^+(\widehat{\frak{gl}}_n)}

\newcommand{\afuglqm}{\sU^-(\widehat{\frak{gl}}_n)}
\newcommand{\afuglz}{\sU_\mbz(\widehat{\frak{gl}}_n)}
\newcommand{\afuglzp}{\sU_\mbz^+(\widehat{\frak{gl}}_n)}
\newcommand{\afuglzm}{\sU_\mbz^-(\widehat{\frak{gl}}_n)}
\newcommand{\afuglzz}{\sU_\mbz^0(\widehat{\frak{gl}}_n)}

\newcommand{\afSrmbz}{{\mathcal S}_{\vtg}(n,r)_{\mathbb Z}}
\newcommand{\afSrmbq}{{\mathcal S}_{\vtg}(n,r)_{\mathbb Q}}
\newcommand{\afSrk}{{\mathcal S}_{\vtg}(n,r)_{\field}}

\newcommand{\affSr}{{\fS_{\vtg,r}}}

\newcommand{\afgl}{\widehat{\frak{gl}}_n}

\newcommand{\afbse}{\boldsymbol e^\vartriangle}

\newcommand{\afE}{E^\vartriangle}

\newcommand{\afmbnn}{\mathbb N_\vtg^{n}}
\newcommand{\afmbzn}{\mathbb Z_\vtg^{n}}

\newcommand{\afThn}{\Theta_\vtg(n)}
\newcommand{\aftiThn}{\widetilde\Theta_\vtg(n)}
\newcommand{\afThnpm}{\Theta_\vtg^\pm(n)}
\newcommand{\afThnp}{\Theta_\vtg^+(n)}
\newcommand{\afThnm}{\Theta_\vtg^-(n)}

\newcommand{\afThnr}{\Theta_\vtg(n,r)}

\newcommand{\afMnq}{M_{\vtg,n}(\mathbb Q)}
\newcommand{\afMnz}{M_{\vtg,n}(\mathbb Z)}

\newcommand{\afLa}{\Lambda_\vtg}
\newcommand{\afLanr}{\Lambda_\vtg(n,r)}

\newcommand{\tri}{\triangle(n)}

\newcommand{\Uz}{\ti\sU_\mbz(\afgl)}
\newcommand{\msE}{\mathscr E}

\newcommand{\msX}{\mathscr X}
\newcommand{\msEap}{\mathscr E^{(A^+)}}
\newcommand{\msEa}{\mathscr E^{(A)}}
\newcommand{\msFaf}{\mathscr F^{(A^-)}}
\newcommand{\msFa}{\mathscr F^{(A)}}

\begin{document}
\title{On the hyperalgebra of the loop algebra $\afgl$}

\author{Qiang Fu}
\address{Department of Mathematics, Tongji University, Shanghai, 200092, China.}
\email{q.fu@hotmail.com, q.fu@tongji.edu.cn}


\thanks{Supported by the National Natural Science Foundation
of China}

\begin{abstract}
Let $\Uz$ be the Garland integral form of $\afuglq$ introduced by Garland \cite{Ga}, where $\afuglq$ is the universal enveloping algebra of $\afgl$.
Using Ringel--Hall algebras, a certain integral form $\afuglz$ of  $\afuglq$ was constructed in \cite{Fu13}.
We prove that the Garland integral form $\Uz$ coincides with $\afuglz$.
Let $\field$ be a commutative ring with unity and let $\afuglk=\afuglz\ot\field$. For $h\geq 1$, we use Ringel--Hall algebras to construct a certain subalgebra, denoted by $\Unkh$, of $\afuglk$. The algebra $\Unkh$ is the affine analogue of $\ttu(\frak{gl}_n)_h$, where $\ttu(\frak{gl}_n)_h$ is a certain subalgebra of the hyperalgebra associated with $\frak{gl}_n$ introduced by Humhpreys \cite{Hum}.
The algebra $\ttu(\frak{gl}_n)_h$ plays an important role in the modular representation theory of $\frak{gl}_n$. In this paper we give a realization of $\Unkh$ using affine Schur algebras.
\end{abstract}

 \sloppy \maketitle
\section{Introduction}
The quantum groups are usually defined by presentations with generators and relations. In  a famous paper  \cite{Ri90}, Ringel discovered the Hall algebra realization  of the positive part of the quantum groups associated with finite dimensional semisimple complex Lie algebras.
Using a geometric setting of $q$-Schur algebras,
Beilinson--Lusztig--MacPherson (BLM) gave a geometric interpretation for the entitle quantum $\frak{gl}_n$ in \cite{BLM}. The geometric interpretation  is in terms of flags in a vector space over a finite field. They first use $q$-Schur algebras to construct a certain algebra $\bfKn$ without unity, and realize quantum $\frak{gl}_n$ as a subalgebra of the completion algebra $\hbfKn$.
The remarkable BLM's work has many important applications. Based on BLM's work, Du \cite{Du95} established partially the quantum Schur--Weyl duality at the integral level.
The algebra $\bfKn$ was later generalized by Lusztig to other types, which is called modified quantum groups  (cf. \cite[Ch.~23]{Lubk}). Note that the algebra $\bfKn$ is isomorphic to the modified quantum group of $\frak{gl}_n$ (see \cite[Th.~6.3]{DF09}). 
The categorification of the modified quantum group of $\frak{sl}_n$ was given in \cite{Lauda,KL}.



Like the $q$-Schur algebra, the affine $q$-Schur algebra has several equivalent definitions (see \cite{GV,Gr99,Lu99}). Using affine $q$-Schur algebras,
the quantum  loop algebra of $\mathfrak {gl}_n$ was realized in \cite{DF13} (see also \cite[5.5(2)]{DF10} and \cite[Ch.~6]{DDF}).  These results have important applications to the investigation of the integral affine quantum Schur--Weyl reciprocity (cf. \cite{Fu13,Fu12,DF14}).

Let $G$ be a simply connected semisimple algebraic group with the root system $\Phi$ over an algebraically closed field $\mbf$ of characteristic $p>0$.  Let
$\fkg$ be the complex semisimple Lie algebra corresponding to $\Phi$.
B. Kostant \cite{Ko} introduced a $\mbz$-form $\sU_\mbz(\fkg)$ of the universal enveloping algebra of $\fkg$.
The algebra $\sU_\mbf(\fkg)=\sU_\mbz(\fkg)\ot\mbf$ is isomorphic to the hyperalgebra of $G$.  For each $h\geq 1$, Humphreys \cite{Hum} constructed a certain subalgebra $\ttu(\fkg)_h$ of $\sU_\mbf(\fkg)$. The algebra $\ttu(\fkg)_1$ was originally introduced  by Curtis\cite{Curtis}, as the restricted universal enveloping algebra of the Lie algebra of $G$. The algebra $\ttu(\fkg)_1$ plays an important role in the theory of modular representation of Lie algebras. The behavior of finite dimensional $\ttu(\fkg)_h$-modules is entirely analogous to that of
$\ttu(\fkg)_1$-modules.
A theory of the quantum version of $\ttu(\fkg)_1$ was developed by Lusztig \cite{Lu90}.

Let $\sU_\mbz(\afgl)$ be the $\mbz$-form of $\sU(\afgl)$ introduced in \cite{Fu13}, where $\sU(\afgl)$ is the universal enveloping algebra of $\afgl$.
In \ref{basis2 for afuglz}, we will construct several bases of $\sU_\mbz(\afgl)$. The basis $\fM$ of $\sU_\mbz(\afgl)$ given in \ref{basis2 for afuglz}(1) is the affine analogue of the PBW basis of $\sU_\mbz(\frak{gl}_n)$ constructed by Kostant \cite{Ko}.


Given a complex semisimple Lie algebra $\frak{g}$, Garland \cite{Ga} introduced an integral form of $\sU(\h{\frak{g}})$, where $\sU(\h{\frak{g}})$ is the universal enveloping algebra of the loop algebra $\h{\frak{g}}$.
Let $\Uz$ be the Garland integral form of $\sU(\afgl)$.
Certain mysterious elements $\Psi_{i,l}(\La_k)\in\afuglq$ corresponding to imaginary roots were constructed by Garland \cite[(5.6)]{Ga}. These elements
are essential for the construction of the integral basis for $\Uz$ given in \cite[Th. 5.8]{Ga}. In \ref{key for integral form}, the elements $\Psi_{i,l}(\La_k)$ will be interpreted in terms of Ringel--Hall algebras. Using \ref{key for integral form}, we will prove in \ref{identification} that
$\afuglz$ coincides with $\Uz$.

Let $\field$ be a commutative ring with unity.
Assume $p=\text{char}\field>0$.
We call $\sU_\field(\afgl):=\sU_\mbz(\afgl)\ot\field$ the hyperalgebra of $\afgl$.
For $h\geq 1$ we will use Ringel--Hall algebras to construct
a certain subalgebra $\ttu_\vtg(n)_h$ of $\sU_\field(\afgl)$, which is the affine analogue of $\ttu(\frak{gl}_n)_h$. The quantum version of $\ttu(\frak{gl}_n)_1$ was realized by Beilinson--Lusztig--MacPherson \cite{BLM} using $q$-Schur algebras. We will use affine Schur algebras to give a realization of $\Unkh$ in \ref{vi}. More precisely, we will use affine Schur algebras to construct a certain algebra $\msKnhq$, which is the affine analogue of the algebra $\ms K'$ constructed by Beilinson--Lusztig--MacPherson in \cite[\S6]{BLM}. We will prove in \ref{vi} that $\Unkh$ is isomorphic to $\msKnhq$.

Infinitesimal Schur algebras are certain important subalgebras of Schur algebras introduced by Doty--Nakano--Peters \cite{DNP96}. A theory of quantum version of infinitesimal Schur algebras were investigated by Cox \cite{Cox,Cox00}.
In \cite{DFW05,Fu07}, the little $q$-Schur algebras were introduced
as subalgebras of $q$-Schur algebras and its algebraic structure was
investigated. The relation between infinitesimal and little $q$-Schur algebras was investigated in \cite{Fu05}.
We will investigate the affine version of infinitesimal and little Schur algebras in \S7.


\begin{Not}\label{Notaion} \rm
For a positive integer $n$, let
$\afMnq$ (resp., $\afMnz$) be the set of all matrices
$A=(a_{i,j})_{i,j\in\mbz}$ with $a_{i,j}\in\mbq$ (resp., $a_{i,j}\in\mbz$) such that
\begin{itemize}
\item[(a)]$a_{i,j}=a_{i+n,j+n}$ for $i,j\in\mbz$; \item[(b)] for
every $i\in\mbz$, both sets $\{j\in\mbz\mid a_{i,j}\not=0\}$ and
$\{j\in\mbz\mid a_{j,i}\not=0\}$ are finite.
\end{itemize}
Let $\aftiThn=\{A\in\afMnq\mid a_{i,j}\in\mbn,\,\forall i\not=j,\,
a_{i,i}\in\mbz,\,\forall i\}$.
Let
$\afThn=\{A\in\afMnz\mid a_{i,j}\in\mbn,\,\forall i,j\}$.

Let $\afmbzn=\{(\la_i)_{i\in\mbz}\mid
\la_i\in\mbz,\,\la_i=\la_{i-n}\ \text{for}\ i\in\mbz\}\text{ \,\,
and \,\,} \afmbnn=\{(\la_i)_{i\in\mbz}\in \afmbzn\mid \la_i\ge0\text{ for  }i\in\mbz\}.$
There is a natural order relation $\leq$ on $\afmbzn$ defined by
\begin{equation}\label{order on afmbzn}
\la\leq\mu  \iff\la_i\leq \mu_i\text{ for all $i$}.
\end{equation}
We say that $\la<\mu$ if $\la\leq\mu$ and $\la\not=\mu$.

Let $\sZ=\mbz[\up,\up^{-1}]$, where $\up$ is an indeterminate. Specializing $\up$ to $1$, $\mbq$ and
$\mbz$ will be viewed as $\sZ$-modules.
Let $\field$ be a commutative ring with unity.
Assume $p=\text{char}\field>0$.

\end{Not}

\section{The BLM realization of $\sU(\afgl)$}
Let $\afuglq$ be the universal enveloping algebra of the loop
algebra $\afgl$, where  $\afgl:={\frak{gl}_n}(\mbq)\ot\mbq[\ttt,\ttt^{-1}]$.
For $i,j\in\mbz$, let $\afE_{i,j}\in\afMnq$ be the matrix
$(e^{i,j}_{k,l})_{k,l\in\mbz}$ defined by
\begin{equation*}e_{k,l}^{i,j}=
\begin{cases}1&\text{if $k=i+sn,l=j+sn$ for some $s\in\mbz$;}\\
0&\text{otherwise}.\end{cases}
\end{equation*}
Clearly, the map
$$\afMnq\lra\afgl ,\,\,\,\afE_{i,j+ln}\longmapsto E_{i,j}\ot \ttt^l, \,1\le i,j\le n,l\in\mbz $$
is a Lie algebra isomorphism. We will identify
$\afgl$ with $\afMnq$ in the sequel.

The positive part of $\afuglq$ can be interpreted in terms of Ringel--Hall algebras, which we now describe.
Let $\tri$ ($n\geq 2$) be
the cyclic quiver
with vertex set $I=\mbz/n\mbz=\{1,2,\ldots,n\}$ and arrow set
$\{i\to i+1\mid i\in I\}$. Let $\mathbb F$ be a field. For $i\in I$, let $S_i$
be the corresponding irreducible representation of $\tri$ over $\mathbb F$.
Let
$$\afThnp=\{A\in\afThn\mid a_{i,j}=0\text{ for }i\geq j\}.$$
For $A\in\afThnp$, let
$$M(A)=M_{\mathbb F}(A)=\bop_{1\leq i\leq n\atop i<j,\,j\in\mbz}a_{i,j}M^{i,j},$$
where
$M^{i,j}$ is the unique indecomposable representation for $\tri$ of length $j-i$ with top $S_i$.

By \cite{Ri93}, for $A,B,C\in\afThnp$,
let $\vi^{C}_{A,B}\in\mbz[\up^2]$ be the Hall polynomials such
that, for any finite field ${\mathbb F}_q$,
$\vi^{C}_{A,B}|_{\up^2=q}$ is equal to the number of submodules $N$ of
$M_{{\mathbb F}_q}(C)$ satisfying $N\cong M_{{\mathbb F}_q}(B)$ and $M_{{\mathbb F}_q}(C)/N\cong M_{{\mathbb F}_q}(A)$. For $\la,\mu\in\afmbzn$, we set $\lan\la,\mu\ran=\sum_{1\leq i\leq n}\la_i\mu_i-\sum_{1\leq i\leq n}\la_i\mu_{i+1}$.

Let $\sZ=\mbz[\up,\up^{-1}]$, where $\up$ is an indeterminate.
Let $\Hall$ be the {\it (generic) Ringel--Hall algebra}
of $\tri$, which is by definition the free module over
$\sZ$ with basis $\{u_A\mid A\in\afThnp\}$. The
multiplication is given by
$$u_{A}u_{B}=\up^{\lan \bfd(A),\bfd(B)\ran}\sum_{C\in\afThnp}\vi^{C}_{A,B}(\up^2)u_{C}$$
for $A,B\in \afThnp$, where $\bfd(A)\in\mbn I$ is the dimension vector of $M(A)$.

Let $\Hall_\mbq=\Hall\ot_\sZ\mbq$, where
$\mbq$ is regarded as a $\sZ$-module by specializing $\up$ to $1$.
For $A\in\afThnp$ let $$u_{A,1}=u_A\ot 1\in\Hall_\mbq.$$
\begin{Lem}[{\cite[6.1.4]{DDF}}]\label{th+,th-}
There is a unique injective algebra homomorphism
$\th^+:\Hall_\mbq\ra\afuglq$ (resp., $\th^-:\Hall_\mbq^{\mathrm {op}}\ra\afuglq$) taking $u_{{\afE_{i,j},1}}\mapsto\afE_{i,j}$ (resp.,$u_{{\afE_{i,j},1}}\mapsto\afE_{j,i}$) for all $i<j$.
\end{Lem}

Let $\afuglqp=\th^+(\Hall_\mbq)$ and $\afuglqm=\th^-(\Hall_\mbq^{\mathrm {op}})$.
Let $\afuglqz$ be the subalgebra of $\afuglq$ generated
by $H_i$ for $1\leq i\leq n$, where $H_i=\afE_{i,i}$. Then we have $\afuglq=\afuglqp\ot\afuglqz\ot\afuglqm$.

We now recall the definition of affine Schur algebras.
Let $\affSr$ be the extended affine symmetric group consisting of all permutations
$w:\mbz\ra\mbz$ satisfying $w(i+r)=w(i)+r$ for $i\in\mbz$.
For $r\in\mbn$ let
$\afLa(n,r)=\{\la\in\afmbnn\mid\sum_{1\leq i\leq n}\la_i=r\}.$
For $\la\in\La_\vtg(n,r)$ let  $\fS_\la:=\fS_{(\la_1,\ldots,\la_n)}$
be the corresponding standard Young subgroup of
the symmetric group $\fS_r$.
For a finite subset $X\han\affSr$ and $\la\in\afLanr$, let
$$\ul X=\sum_{x\in
X}x\in\mbq\affSr.$$ The endomorphism algebras over $\mbz$ or $\mbq$
$$\afSrmbz:=\End_{\mbz\affSr}\biggl
(\bop_{\la\in\afLa(n,r)}\mbz\ul{\fS_\la}\affSr\biggr)\,\text{ and }\,\afSrmbq:=\End_{\mbq\affSr}\biggl
(\bop_{\la\in\afLa(n,r)}\mbq\ul{\fS_\la}\affSr\biggr)$$
are called affine Schur algebras (cf. \cite{GV,Gr99,Lu99}).
Note that $\afSrmbq\cong\afSrmbz\ot_\mbz\mbq$.

For $\la,\mu\in\afLanr$, let
$\afmsD_\la=\{d\mid d\in\affSr,\ell(wd)=\ell(w)+\ell(d)\text{ for
$w\in\fS_\la$}\}$ and
$\afmsD_{\la,\mu}=\afmsD_{\la}\cap{\afmsD_{\mu}}^{-1}$.
By \cite[(3.2.1.5)]{DDF} we have
\begin{equation}\label{minimal coset representative}
\aligned
d^{-1}\in\afmsD_\la
&\iff d(\la_{0,i-1}+1)<d(\la_{0,i-1}+2)<\cdots<d(\la_{0,i-1}+\la_i),\,\forall 1\leq i\leq n.\endaligned
\end{equation}
For $1\leq i\leq n$ and $k\in\mbz$ let
\begin{equation}\label{set R}
R_{i+kn}^{\la}=\{\la_{k,i-1}+1,\la_{k,i-1}+2,\ldots,\la_{k,i-1}+\la_i=\la_{k,i}\},
\end{equation}
where $\la_{k,i-1}=kr+\sum_{1\leq t\leq i-1}\la_t$.
Let $$\afThnr=\{A\in\afThn\mid\sg(A)=r\},$$ where $\sg(A)=\sum_{1\leq i\leq n,\,j\in\mbz}a_{i,j}$.
By \cite[\S7.4]{VV99} (see also \cite[Lem.~9.2]{DF10}), there is
a bijective map
\begin{equation*}
{\jmath_\vtg}:\{(\la, d,\mu)\mid
d\in\afmsD_{\la,\mu},\la,\mu\in\afLanr\}\lra\afThnr
 \end{equation*}
sending $(\la, w,\mu)$ to $A=(a_{k,l})$, where $a_{k,l}=|R_k^\la\cap wR_l^\mu|$ for all $k,l\in\mbz$.

For $A=\jmath_\vtg(\la,d,\mu)\in\afThnr$ with $\la,\mu\in\afLanr$, $d\in\afmsD_{\la,\mu}$, define
$[A]_1\in\afSrmbz$ by
\begin{equation*}\label{def of standard basis}
[A]_1(\ul{\fS_\nu} h)=\dt_{\mu\nu}\ul{\fS_\la
d\fS_\mu}h\
\end{equation*}
where $\nu\in\afLanr$ and $h\in\mbz\affSr$. Then by \cite[Th.~2.2.4]{Gr99} the set
$\{[A]_1\mid A\in\afThnr\}$ forms a $\mbz$-basis for $\afSrmbz$.

The algebra $\afuglq$ and the affine Schur algebra $\afSrmbq$ are related by a surjective algebra homomorphism $\xi_r$, which we now describe.
Let
$$\afThnpm=\{A\in\afThn\mid a_{i,i}
=0\text{ for all $i$}\}.$$
For $A\in\afThnpm,{\bf j}\in \afmbnn$ let
$$A[{\bf j},r]= \sum_{\mu\in\La_\vtg(n,r-\sg(A))}
\mu^\bfj[A+\diag(\mu)]_1\in\afSrmbq.
$$
For $i\in\mbz$, let
$\afbse_i\in\afmbnn$ be the element defined by
\begin{equation*}
(\afbse_i)_j=
\begin{cases}
1&\text{if $j\equiv i \mnmod n$}\\
0&\text{otherwise}.
\end{cases}
\end{equation*}
By \cite[Th.~6.1.5]{DDF},
there is an algebra homomorphism
\begin{equation}\label{xir}
\xi_r:\afuglq\ra\afSrmbq
\end{equation}
such that $\xi_r(\afE_{i,j})=\afE_{i,j}[\bfl,r]$ and
$\xi_r(H_i)=0[\afbse_i,r]$
for $i\not=j$.
The maps $\xi_r$ induce an algebra homomorphism
\begin{equation}\label{zeta}
\xi=\prod_{r\geq 0}\xi_r:\sU(\afgl)\ra\prod_{r\geq 0}\afSrmbq.
\end{equation}

For $A\in\afThnpm$ and $\bfj\in\afmbnn$, we set
$$A[\bfj]=(A[\bfj,r])_{r\geq 0}\in\prod_{r\geq 0}\afSrmbq.$$
Let $\afVn_\mbq$ be the $\mbq$-subspace of $\prod_{r\geq 0}\afSrmbq$ spanned by
$A[\bfj]$ for $A\in\afThnpm$, $\bfj\in\afmbnn$.

\begin{Thm}[{\cite[Th.~6.3.4]{DDF}}]\label{realization}
The $\mbq$-space $\afVn_\mbq$ is a subalgebra of $\prod_{r\geq 0}\afSrmbq$ with $\mbq$-basis $\{A[\bfj]\mid A\in\afThnpm,\,\bfj\in\afmbnn\}$. Furthermore, the map $\xi$ is injective and induces a $\mbq$-algebra isomorphism $\xi:\sU(\afgl)\ra\afVn_\mbq$.
\end{Thm}
We shall identify $\sU(\afgl)$ with $\afVn_\mbq$ via $\xi$ and hence identify $\afE_{i,j}$ with $\afE_{i,j}[\bfl]$, etc., in the sequel.

\section{Multiplication formulas in affine Schur algebras}
Some multiplication formulas in affine Schur algebras were given in \cite[Prop.~6.2.3]{DDF}. These formulas are very important for the realization of $\afuglq$ given in \cite[Ch.~6]{DDF}. However, these formulas are not enough for the investigation of $\Unkh$. We will generalize these formulas to a more general setting in \ref{formula}.



We need the following simple lemma.

\begin{Lem}\label{double coset}
For $\la,\mu\in\afLanr$ and $d\in\msD_{\la,\mu}^\vtg$ with
$\jmath_\vtg(\la,d,\mu)=A\in\afThnr$, let $\nu^{(i)}$ be the composition of $\la_i$ obtained by
 removing all zeros from row $i$ of $A$. Then $\frak S_\la
\cap d\frak S_\mu d^{-1}=\frak S_\nu$, where
$\nu=(\nu^{(1)},\ldots,\nu^{(n)})$.
\end{Lem}

By \cite[(6.4.0.1)]{DDF}, we have
\begin{equation}\label{conjugate intersection}
\ul{{\frak S}_\la} w\ul{\frak S_\mu}=|w^{-1} \frak S_\la
w\cap\frak S_\mu|\ul{{\frak S}_\la w\frak S_\mu}
\end{equation}
for $\la,\mu\in\afLanr$ and $w\in\affSr$.

If $\{i_1,i_2,\cdots,i_t\}$ is a subset of $\{1,2,\cdots,r\}$ and $w\in\affSr$ is such that $w(s)=s$ for $s\in\{1,2,\cdots,r\}\backslash\{i_1,i_2,\cdots,i_t\}$ then the element $w$ will be written as
$$\left(\begin{matrix}
i_1&i_2&\cdots&i_t\\
w(i_1)&w(i_2)&\cdots&w(i_t)\\
\end{matrix}\right).$$

Suppose $\la\in\afLa(n,r)$.
Let $k\geq 1$, $1\leq h\leq n$ and $j=a+mn\in\mbz$ with $j\not=h$, $1\leq a\leq n$ and $m\in\mbz$.
If $h<a$ and $m\geq 0$ let
$$u_{a,m,h,k}^\la\!=\!\left(\begin{matrix}
\la_{0,h}+1\!&\!\cdots\!&\!\la_{0,a-1}  &\la_{0,a-1}+1&\cdots& \la_{0,a-1}+k\\
\la_{0,h}+1+k&\cdots&\la_{0,a-1}+k&
\la_{0,h}-mr+1&\cdots& \la_{0,h}-mr+k\\
\end{matrix}\right).$$
If $h\leq a$ and $m<0$ let
$$u_{a,m,h,k}^\la\!=\!\left(\begin{matrix}
\la_{0,h-1}+1\!\!\!   &  \cdots\!\!& \la_{0,a}-k  \!\!\!&\la_{0,a}-k+1\!\!&\cdots\!\!& \la_{0,a}\\
\la_{0,h-1}+1+k\!\!\! & \cdots\!\!& \la_{0,a}
\!\!\!&
\la_{0,h-1}-mr+1\!\!&\cdots\!\!& \la_{0,h-1}-mr+k\\
\end{matrix}\right).$$
If $h\geq a$ and $m> 0$ let
$$u_{a,m,h,k}^\la\!=\!\left(\begin{matrix}
\la_{0,a-1}+1\!&\!\cdots\!&\!\la_{0,a-1}+k  &\la_{0,a-1}+k+1&\cdots& \la_{0,h}\\
\la_{0,h}-mr-k+1&\cdots&\la_{0,h}-mr&
\la_{0,a-1}+1&\cdots& \la_{0,h}-k\\
\end{matrix}\right).$$
If $h>a$ and $m\leq 0$ let
$$u_{a,m,h,k}^\la\!=\!\left(\begin{matrix}
\la_{0,a}-k+1\!&\!\cdots\!&\!\la_{0,a}  &\la_{0,a}+1&\cdots& \la_{0,h-1}\\
\la_{0,h-1}-mr-k+1&\cdots&\la_{0,h-1}-mr&
\la_{0,a}+1-k&\cdots& \la_{0,h-1}-k\\
\end{matrix}\right).$$

Combining \eqref{minimal coset representative} and \ref{double coset},
we obtain the following result.
\begin{Lem}\label{power of e_i}
Maintain the notation introduced above.
\begin{itemize}
\item[(1)]
$u_{a,m,h,k}^\la\in\afmsD_{\la+k\afbse_h-k\afbse_j,\la}$ and $\jmath_\vtg(\la+k\afbse_h-k\afbse_j,u_{a,m,h,k}^\la,\la)=k\afE_{h,j}+\diag(\la-k\afbse_j)$.
\item[(2)]
$\frak S_\beta=(u_{a,m,h,k}^\la)^{-1}\frak S_{\la+k\afbse_h-k\afbse_j} u_{a,m,h,k}^{\la}\cap\fS_\la$  where
\[
\bt=
\begin{cases}
(\la_1,\ldots,\la_{a-1},
k,\la_a-k,\la_{a+1},\ldots,\la_n),\;&\text{if
either $m> 0$ or both $m=0$ and $h<a$;}\\
(\la_1,\ldots,\la_{a-1},\la_a-k,k,\la_{a+1},\ldots,\la_n),&\text{if either $m<0$ or both $m=0$ and $h>a$}.
\end{cases}
\]
\end{itemize}
\end{Lem}


For $i\in\mbz$, let $\bar i$ denote the integer modulo $n$.
For $A\in\aftiThn$, let
$\ro(A)=\bigl(\sum_{j\in\mbz}a_{i,j}\bigr)_{i\in\mbz}$ and
$\co(A)=\bigl(\sum_{i\in\mbz}a_{i,j}\bigr)_{j\in\mbz}.$ Note that if $A,B\in\afThnr$ with $\co(B)\not=\ro(A)$, then $[B]_1\cdot[A]_1=0$ in $\afSrmbz$.

\begin{Prop}\label{formula1}
Let $k\geq 1$, $1\leq h\leq n$ and $j\in\mbz$ with $\bar j\not=\bar h$. Assume $A\in\afThnr$ and $\la=\ro(A)\geq k\afbse_j$.
Then the following identity holds in $\afSrmbz:$
$$[k\afE_{h,j}+\diag(\la-k\afbse_j)]_1\cdot[A]_1
=\sum_{\dt\in\La(\infty,k)\atop a_{j,t}\geq\dt_t,\,\forall t}\prod_{t\in\mbz}\left({a_{h,t}+\dt_t\atop\dt_t}\right)
\bigg[A+\sum_{t\in\mbz}\dt_t(\afE_{h,t}-\afE_{j,t})\bigg]_1,$$
where $\La(\infty,k)=\{(\la_i)_{i\in\mbz}\mid\la_i\in\mbn\  \forall i,\,\sum_{i\in\mbz}\la_i=k\}$.
\end{Prop}
\begin{proof}
Assume $\mu=\co(A)$ and $d\in\msD^\vtg_{\la,\mu}$ such that
$\jmath_\vtg(\la, d,\mu)=A$.
We write $j=a+mn$ with $1\leq a\leq n$ and $m\in\mbz$.
Then by
\ref{double coset}, \eqref{conjugate intersection} and
\ref{power of e_i} we have
\[
\begin{split}
[k\afE_{h,j}+\diag(\la-k\afbse_j)]_1[A]_1(\ul{\frak S_\mu})&=\ul{\frak S_{\la+k\afbse_h-k\afbse_j}u_{a,m,h,k}^\la\frak S_\la}\cdot d\cdot\ul{\msD^\vtg_\al\cap\frak S_\mu}\\
&=\frac{1}{|\frak S_\la|}\ul{\frak S_{\la+k\afbse_h-k\afbse_j}u_{a,m,h,k}^\la\frak
S_\la}\cdot\ul{\frak S_\la d\frak S_\mu}\\
&=\frac{1}{|\frak S_\la|}\prod_{1\leq s\leq n\atop
t\in\mbz}\frac{1}{a_{s,t}!}\ul{\frak S_{\la+k\afbse_h-k\afbse_j}u_{a,m,h,k}^\la\frak
S_\la}\cdot\ul{\frak S_\la}\cdot d\cdot\ul{\frak S_\mu}\\
&=\prod_{1\leq s\leq n\atop t\in\mbz}\frac{1}{a_{s,t}!}\ul{\frak
S_{\la+k\afbse_h-k\afbse_j}u_{a,m,h,k}^\la\frak S_\la}\cdot d\cdot\ul{\frak S_\mu}\\
&=\prod_{1\leq s\leq n\atop t\in\mbz}\frac{1}{a_{s,t}!}\ul{\frak
S_{\la+k\afbse_h-k\afbse_j}}\cdot u_{a,m,h,k}^\la\cdot\ul{\msD^\vtg_\beta\cap\frak
S_\la}\cdot d\cdot\ul{\frak S_\mu},
\end{split}
\]
where $\frak S_\al=d^{-1}\frak S_\la d\cap\frak S_\mu$ and
$\bt$ is as in \ref{power of e_i}(2).

For $w\in\afmsD_\bt\cap\frak S_\la$ let
$B^{(w)}=(b_{s,t}^{(w)})\in\afThnr$, where
$b_{s,t}^{(w)}=|R_s^{\la+k\afbse_h-k\afbse_j}\cap u_{a,m,h,k}^\la wd R_t^\mu|
=|w^{-1}(u_{a,m,h,k}^\la)^{-1}R_s^{\la+k\afbse_h-k\afbse_j}\cap dR_t^\mu|$.
Let
$$Y_{h,a,m}
=\begin{cases}
\{\la_{0,a-1}+1+mr,\cdots,\la_{0,a-1}+k+mr\}\! \!&\text{if either $m>0$ or both $m=0$ and $h<a$;}\\
\{\la_{0,a}-k+1+mr,\cdots,\la_{0,a}+mr\}\!\!&\text{if either $m<0$ or both $m=0$ and $h>a$}.
\end{cases}$$
Then for $w\in\afmsD_\bt\cap\frak S_\la$ we have
$$
w^{-1}(u_{a,m,h,k}^\la)^{-1}R_s^{\la+k\afbse_h-k\afbse_j}=
\begin{cases}
R_s^\la&\text{if $1\leq s\leq n$ and $\bar s\not=\bar h,\bar j$;} \\
R_h^\la\cup w^{-1}(Y_{h,a,m})&\text{if $s=h$;} \\
R_j^\la\backslash w^{-1}(Y_{h,a,m})&\text{if $s=j$.} \\
\end{cases}
$$
This implies that
$$
b_{s,t}^{(w)}=
\begin{cases}
a_{s,t} &\text{if $1\leq s\leq n$ and $\bar s\not=\bar h,\bar j$;}\\
a_{h,t}+|w^{-1}(Y_{h,a,m})\cap dR_t^{\mu}|&\text{if $s=h$;}\\
a_{j,t}-|w^{-1}(Y_{h,a,m})\cap dR_t^{\mu}|&\text{if $s=j$.}
\end{cases}
$$
Thus by
\ref{double coset} and \eqref{conjugate intersection} we have
\[
\begin{split}
&\qquad [k\afE_{h,j}+\diag(\la-k\afbse_j)]_1[A]_1(\ul{\frak S_\mu})\\
&=\prod_{1\leq s\leq n\atop t\in\mbz}\frac{1}{a_{s,t}!}\sum_{w\in\fS_\la\cap\afmsD_\bt}\prod_{1\leq s\leq n\atop t\in\mbz}b_{s,t}^{(w)}!\ul{\frak
S_{\la+k\afbse_h-k\afbse_j}  u_{a,m,h,k}^\la w d {\frak S_\mu}}\\
&= \sum_{w\in\fS_\la\cap\afmsD_\bt}\prod_{1\leq s\leq n\atop t\in\mbz}\frac{b_{s,t}^{(w)}!}{a_{s,t}!}[B^{(w)}]_1(\ul{\fS_\mu})\\
&=\sum_{\dt\in\La(\infty,k)} \sum_{w\in\fS_\la\cap\afmsD_\bt
\atop\dt_t=|w^{-1}(Y_{h,a,m})\cap dR_t^\mu|,\,\forall t}\prod_{1\leq s\leq n\atop t\in\mbz}\frac{b_{s,t}^{(w)}!}{a_{s,t}!}[B^{(w)}]_1(\ul{\fS_\mu})\\
&=\sum_{\dt\in\La(\infty,k)\atop a_{j,t}\geq\dt_t,\,\forall t}|\mpX_\dt|\prod_{t\in\mbz}\frac{(a_{h,t}+\dt_t)!
(a_{j,t}-\dt_t)!}{a_{h,t}!a_{j,t}!}
\bigg[A+\sum_{t\in\mbz}\dt_t(\afE_{h,t}-\afE_{j,t})\bigg]_1
(\ul{\fS_\mu}),
\end{split}
\]
where $\mpX_\dt=\{w\in\fS_\la\cap\msD_\bt\mid\dt_t=|w^{-1}(Y_{h,a,m})\cap
dR_t^\mu|,\,\forall t\in\mbz\}$. For $\dt\in\La(\infty,k)$
let $\mpY_\dt=\{Y\mid Y\han R_j^\la,\,|Y|=k,\,|Y\cap dR_t^\mu|=\dt_t,\,\forall t\in\mbz\}$ and $\mpZ_\dt=\{(Z_t)_{t\in\mbz}\mid |Z_t|=\dt_t,\,Z_t\han
R_j^\la\cap dR_t^\mu,\,\forall t\in\mbz\}$.
Define
$$\begin{array}{rl}
g_\dt:\mpX_\dt&\lra \mpY_\dt,\\
 w&\longmapsto w^{-1}(Y_{h,a,m}),
 \end{array}
 \quad
\begin{array}{rl}
h_\dt:\mpY_\dt&\lra \mpZ_\dt,\\
Y&\longmapsto (Y\cap dR_t^\mu)_{t\in\mbz}.
 \end{array}
 $$
Then the maps $g_\dt$ and $h_\dt$ are all bijective (cf. \cite[Lem.~5.2]{Fu12}).
It follows that $|\mpX_\dt|=|\mpY_\dt|=|\mpZ_\dt|=\prod_{t\in\mbz}\big({a_{j,t}\atop\dt_t}\big)$.
Consequently, we have
\[
\begin{split}
&\qquad[k\afE_{h,j}+\diag(\la-k\afbse_j)]_1[A]_1\\
&=\sum_{\dt\in\La(\infty,k)\atop a_{j,t}\geq\dt_t,\,\forall t}\prod_{t\in\mbz}\left({a_{j,t}\atop\dt_t}\right)\frac{(a_{h,t}+\dt_t)!
(a_{j,t}-\dt_t)!}{a_{h,t}!a_{j,t}!}
\bigg[A+\sum_{t\in\mbz}\dt_t(\afE_{h,t}-\afE_{j,t})\bigg]_1\\
&=\sum_{\dt\in\La(\infty,k)\atop a_{j,t}\geq\dt_t,\,\forall t}\prod_{t\in\mbz}\left({a_{h,t}+\dt_t
\atop\dt_t}\right)
\bigg[A+\sum_{t\in\mbz}\dt_t(\afE_{h,t}-\afE_{j,t})\bigg]_1
\end{split}
\]
as required.
\end{proof}

\begin{Prop}\label{formula2}
Let $k\geq 1$, $1\leq h\leq n$ and $m\not=0\in\mbz$. Assume $A\in\afThnr$ and $\la=\ro(A)\geq k\afbse_h$. Then the following identity holds in $\afSrmbz:$
\[
\begin{split}
&\qquad[k\afE_{h,h+mn}+\diag(\la-k\afbse_h)]_1\cdot[A]_1\\
&=\sum_{\dt\in\La(\infty,k)\atop a_{h+mn,t}-\dt_t+\dt_{t-mn}\geq 0,\,\forall t}\prod_{t\in\mbz}\left({a_{h,t}+\dt_t-\dt_{t+mn}\atop\dt_t}\right)
\bigg[A+\sum_{t\in\mbz}\dt_t(\afE_{h,t}-\afE_{h+mn,t})\bigg]_1.
\end{split}
\]
\end{Prop}
\begin{proof}
Let $u_{m,h,k}^\la=u_{h,m,h,k}^\la$.
Assume $\mu=\co(A)$ and $d\in\msD^\vtg_{\la,\mu}$ such that
$\jmath_\vtg(\la, d,\mu)=A$.
Let $j=h+mn$. Then by
\ref{double coset}, \eqref{conjugate intersection} and
\ref{power of e_i} we have
\[
\begin{split}
[k\afE_{h,j}+\diag(\la-k\afbse_j)]_1[A]_1(\ul{\frak S_\mu})
&=\prod_{1\leq s\leq n\atop t\in\mbz}\frac{1}{a_{s,t}!}\ul{\frak
S_{\la}}\cdot u_{m,h,k}^\la\cdot\ul{\msD^\vtg_\beta\cap\frak
S_\la}\cdot d\cdot\ul{\frak S_\mu},
\end{split}
\]
where
$\bt$ is as in \ref{power of e_i}(2).

For $w\in\afmsD_\bt\cap\frak S_\la$ let
$B^{(w)}=(b_{i,j}^{(w)})\in\afThnr$, where
$b_{s,t}^{(w)}=|R_s^{\la}\cap u_{m,h,k}^\la wd R_t^\mu|
=|w^{-1}(u_{m,h,k}^\la)^{-1}R_s^{\la}\cap dR_t^\mu|$.
Then we have
$$
(u_{m,h,k}^\la)^{-1}R_s^{\la}=
\begin{cases}
R_s^\la&\text{if $1\leq s\leq n$ and $ s\not= h$;} \\
(R_h^\la\backslash(Y_{m}-mr))\cup Y_m&\text{if $s=h$,} \\
\end{cases}
$$
where
$$Y_{m}
=\begin{cases}
\{\la_{0,h-1}+1+mr,\cdots,\la_{0,h-1}+k+mr\}\! \!&\text{if $m>0$}\\
\{\la_{0,h}-k+1+mr,\cdots,\la_{0,h}+mr\}\!\!&\text{if $m<0$}.
\end{cases}$$
It follows that
$$
b_{s,t}^{(w)}=
\begin{cases}
a_{s,t} &\text{if $1\leq s\leq n$ and $ s\not= h $;}\\
a_{h,t}-|w^{-1}(Y_{m}-mr)\cap dR_t^{\mu}|+
|w^{-1}(Y_{m})\cap dR_t^{\mu}|&\text{if $s=h$.}
\end{cases}
$$
This together with
\ref{double coset} and \eqref{conjugate intersection} gives
\[
\begin{split}
&\qquad[k\afE_{h,j}+\diag(\la-k\afbse_j)]_1[A]_1(\ul{\frak S_\mu})\\
&= \sum_{w\in\fS_\la\cap\afmsD_\bt}\prod_{1\leq s\leq n\atop t\in\mbz}\frac{b_{s,t}^{(w)}!}{a_{s,t}!}[B^{(w)}]_1(\ul{\fS_\mu})\\
&=\sum_{\dt\in\La(\infty,k)
\atop a_{j,t}-\dt_t+\dt_{t-mn}\geq 0,\,\forall t}|\mpX_\dt|\prod_{t\in\mbz}
\frac{(a_{h,t}+\dt_t-\dt_{t+mn})!
}{a_{h,t}!}
\bigg[A+\sum_{t\in\mbz}\dt_t(\afE_{h,t}-\afE_{j,t})\bigg]_1
(\ul{\fS_\mu}),
\end{split}
\]
where $\mpX_\dt=\{w\in\fS_\la\cap\msD_\bt\mid\dt_t=|w^{-1}(Y_{m})\cap
dR_t^\mu|,\,\forall t\in\mbz\}$. For $\dt\in\La(\infty,k)$
there is a bijective map $f_\dt:\mpX_\dt\ra\mpZ_\dt$
defined by sending $w$ to $(w^{-1}(Y_m)\cap dR_t^\mu)_{t\in\mbz}$,
where $\mpZ_\dt=\{(Z_t)_{t\in\mbz}\mid |Z_t|=\dt_t,\,Z_t\han
R_j^\la\cap dR_t^\mu,\,\forall t\in\mbz\}$
(cf. \cite[Lem.~5.2]{Fu12}).
This implies that $$|\mpX_\dt|=|\mpZ_\dt|=\prod_{t\in\mbz}\bigg({a_{j,t}\atop\dt_t}\bigg)
=\prod_{t\in\mbz}\frac{a_{h,t}!}{\dt_t!(a_{h,t}-\dt_{t+mn})!}.
$$
Hence we have
\[
\begin{split}
&\qquad[k\afE_{h,j}+\diag(\la-k\afbse_j)]_1[A]_1\\
&=\sum_{\dt\in\La(\infty,k)
\atop a_{j,t}-\dt_t+\dt_{t-mn}\geq 0,\,\forall t}\prod_{t\in\mbz}
\left({a_{h,t}+\dt_t-\dt_{t+mn}\atop\dt_t}\right)
\bigg[A+\sum_{t\in\mbz}\dt_t(\afE_{h,t}-\afE_{j,t})\bigg]_1.
\end{split}
\]
The proof is completed.
\end{proof}

Combining \ref{formula1} and \ref{formula2}, we obtain the following
result.
\begin{Prop}\label{formula}
Let $k\geq 1$, and $i,j\in\mbz$ with $i\not=j$. Assume $A\in\afThnr$ and $\la=\ro(A)\geq k\afbse_j$. Then the following identity holds in $\afSrmbz:$
\[
\begin{split}
&\qquad[k\afE_{i,j}+\diag(\la-k\afbse_j)]_1\cdot[A]_1\\
&=\sum_{\dt\in\La(\infty,k)\atop a_{j,t}-\dt_t+\dt_{\bar i,\bar j}\dt_{t+i-j}\geq 0,\,\forall t}\prod_{t\in\mbz}\left({a_{i,t}+\dt_t-\dt_{\bar i,\bar j}\dt_{t+j-i}\atop\dt_t}\right)
\bigg[A+\sum_{t\in\mbz}\dt_t(\afE_{i,t}-\afE_{j,t})\bigg]_1.
\end{split}
\]
\end{Prop}

\section{The algebra $\afuglz$}

Recall from \S2 that $\Hall$ is the Ringel-Hall algebra of $\tri$ over $\sZ$.
Let $\Hall_\mbz=\Hall\ot_\sZ\mbz$, where
$\mbz$ is regarded as a $\sZ$-module by specializing $\up$ to $1$.
Let $\sU_\mbz^+(\afgl)=\th^+(\Hall_\mbz)$ and
$\sU_\mbz^-(\afgl)=\th^-(\Hall_\mbz^{\mathrm{op}})$, where the maps
$\th^+$ and $\th^-$ are defined in \ref{th+,th-}.
For $\la\in\afmbnn$ let
$$\bigg({H\atop\la}\bigg)=\prod_{1\leq i\leq n}\bigg({H_i\atop\la_i}\bigg),$$
where $H_i=\afE_{i,i}$ and
$$\bpa{H_i}{\la_i}=\frac{H_i(H_i-1)\cdots
(H_{i}-\la_i+1)}{\la_i!}.$$
Let $\afuglzz$ be the $\mbz$-submodule of $\afuglq$
spanned by $\big({H\atop\la}\big)$  for  $\la\in\afmbnn$.
Let $\afuglz=\afuglzp\afuglzz\afuglzm.$

For $A\in\afThnp$ let
\begin{equation}
u_{A,1}^+=\th^+(u_{A,1})\in\afuglzp\text{ and }u_{A,1}^-=\th^-(u_{A,1})\in\afuglzm.
\end{equation}
By \cite[Prop.~6.1.4 and Th.~6.1.5]{DDF} we have
\begin{equation}\label{th+uA}
u_{A,1}^+=A\{\bfl\}\text{ and }u_{A,1}^-=(\tA)\{\bfl\},
\end{equation}
where $\tA$ is the transpose matrix of $A$.
The following result is given in \cite[Th.~6.5 and Lem.~6.2]{Fu13}.

\begin{Thm}\label{basis for afuglz}
The $\mbz$-module $\afuglz$ is a $\mbz$-subalgebra of $\afuglq$.
Furthermore, the set $\{u^+_{A,1} \big({H\atop\la}\big)
u^-_{B,1}\mid A,B\in\afThnp,\,\la\in\afmbnn\}$ forms a $\mbz$-basis for $\afuglz$.
\end{Thm}

For $A\in\afThnpm$, $\la\in\afmbnn$ let
\begin{equation*}
\begin{split}
A\br{\la,r}&=\sum_{\mu\in\afLa(n,r-\sg(A))}\bigg({\mu\atop\la}\bigg)
[A+\diag(\mu)]_1\in\afSrmbq,
\end{split}
\end{equation*}
where
$\big({\mu\atop\la}\big)=
\big({\mu_1\atop\la_1}\big)\cdots\big({\mu_n\atop\la_n}\big).$
Also, for $A\in\afMnz$, we set
$A\br{\la,r}=0$ if $a_{i,j}<0$ for some $i\not=j$.
Let $$A\{\la\}=(A\{\la,r\})_{r\geq 0}\in\prod_{r\geq 0}\afSrmbq.$$
Clearly, we have
\begin{equation}\label{prod(Hi lai)}
0\{\la\}=\bigg({H\atop\la}\bigg)\in\afuglz,
\end{equation}
for $\la\in\afmbnn$.

According to \cite[Prop.~7.3(1)]{Fu12} we have the following result (cf. \cite[Lem.~3.4]{Fu15a}).
\begin{Lem}\label{0mu Ala}
Let $A\in\afThnpm$ and $\la,\mu\in\afmbnn$. Then we have the
following formula in $\prod_{r\geq 0}\afSrmbq:$
$$\bigg({H\atop\mu}\bigg)  A\br\la=\sum_{\dt\in\afmbnn\atop\dt\leq\mu}
\bigg({\dt+\la\atop\la}\bigg)
\bigg(
\sum_{\bt\in\afmbnn\atop\bt\leq\mu-\dt,\,\bt\leq\la}
\bigg({\ro(A)\atop\mu-\bt-\dt}
\bigg)\bigg({\la\atop\bt}\bigg)\bigg)A\br{\la+\dt}.$$
\end{Lem}

By \ref{realization} we know that $A\{\bfl\}=A[\bfl]\in\afuglq$.
Furthermore, from \ref{0mu Ala} we see that
\begin{equation}\label{0la A0}
\bigg({H\atop\la}\bigg)  A\{\bfl\}=A\{\la\}+\sum_{\dt\in\afmbnn\atop\dt<\la}
\bigg({\ro(A)\atop\la-\dt}\bigg)A\{\dt\}.
\end{equation}
for $A\in\afThnpm$ and $\la\in\afmbnn$. Consequently, we conclude that
$A\{\la\}\in\afuglq$ for all $A,\la$. We will see in
\ref{basis2 for afuglz} that $A\{\la\}$ belongs to $\afuglz$ for all $A,\la$.

We now use \ref{formula} to establish the following multiplication formulas in
$\afuglq$.

\begin{Prop}\label{fundamental formula}
Let $k\geq 1$ and $i,j\in\mbz$ with $i\not=j$. Assume $A\in\afThnpm$ and $\la\in\afmbnn$. The following multiplication formulas hold in $\afuglq:$
\[
\begin{split}
(k\afE_{i,j})\{\bfl\}A\{\la\}&=\sum_{\al\in\La(\infty,k)\atop \dt\in\afmbnn,\,\dt\leq\la}
a(\al,\dt)A^{(\al)}\{\al_i\afbse_i+\dt\},
\end{split}
\]
where $A^{(\al)}=A+\sum_{t\not=i,\,t\in\mbz}\al_t\afE_{i,t}
-\sum_{t\not=j,\,t\in\mbz}\al_t\afE_{j,t}$ and
$$a(\al,\dt)=\bigg({\dt_i+\al_i\atop\al_i}\bigg)\prod_{t\not=i\atop t\in\mbz}
\left({a_{i,t}+\al_t-\dt_{\bar j,\bar i}\al_{t+j-i}\atop\al_t}\right)
\sum_{0\leq b\leq\min\{\al_i,\la_i-\dt_i\}}\bigg(
{\al_j\afbse_j-\al_i\afbse_i\atop\la-\dt-b\afbse_i}\bigg)
\bigg({\al_i\atop b}\bigg).$$
\end{Prop}
\begin{proof}
By \ref{formula} we have
\[
\begin{split}
(k\afE_{i,j})\{\bfl,r\}A\{\la,r\}
&=\sum_{\ga\in\afmbzn}\bigg({\ga\atop\la}\bigg)
\sum_{\al\in\La(\infty,k)}\prod_{t\not=i\atop t\in\mbz}
\left({a_{i,t}+\al_t-\dt_{\bar j,\bar i}\al_{t+j-i}\atop\al_t}\right)
\left(\ga_i+\al_i-\dt_{\bar j,\bar i}\al_j\atop\al_i\right)\\
&\qquad\times
\bigg[A^{(\al)}+\diag(\ga+\al_i\afbse_i-\al_j\afbse_j)\bigg]_1\\
&=\sum_{\al\in\La(\infty,k)}\prod_{t\not=i\atop t\in\mbz}
\left({a_{i,t}+\al_t-\dt_{\bar j,\bar i}\al_{t+j-i}\atop\al_t}\right)x_\al,
\end{split}
\]
where
$$x_\al=\sum_{\nu\in\afmbzn}
\bigg({\nu-\al_i\afbse_i+\al_j\afbse_j\atop\la}\bigg)
\bigg({\nu_i\atop\al_i}\bigg)
\bigg[A^{(\al)}+\diag(\nu)\bigg]_1.$$
Since
\[
\begin{split}
&\qquad\bigg({\nu-\al_i\afbse_i+\al_j\afbse_j\atop\la}\bigg)
\bigg({\nu_i\atop\al_i}\bigg)\\&=
\sum_{\bfj\in\afmbnn,\,\bfj\leq\la}
\bigg({\al_j\afbse_j-\al_i\afbse_i\atop\la-\bfj}\bigg)
\bigg({\nu\atop\bfj}\bigg)\bigg({\nu_i\atop\al_i}\bigg)\\
&=\sum_{\bfj\in\afmbnn,\,\bfj\leq\la\atop 0\leq b\leq\min\{j_i,\al_i\}}
\bigg({\al_j\afbse_j-\al_i\afbse_i\atop\la-\bfj}\bigg)
\bigg({j_i+\al_i-b\atop b,\,j_i-b,\,\al_i-b}\bigg)\bigg(
{\nu\atop\bfj+(\al_i-b)\afbse_i}\bigg),
\end{split}
\]
we conclude that
\[
\begin{split}
x_\al
&=\sum_{\bfj\in\afmbnn,\,\bfj\leq\la\atop 0\leq b\leq\min\{j_i,\al_i\}}
\bigg({\al_j\afbse_j-\al_i\afbse_i\atop\la-\bfj}\bigg)
\bigg({j_i+\al_i-b\atop b,\,j_i-b,\,\al_i-b}\bigg)A^{(\al)}\{\bfj+(\al_i-b)\afbse_i,r\}.\\
&=\sum_{\dt\in\afmbnn,\,\dt\leq\la\atop
0\leq b\leq\min\{\la_i-\dt_i,\al_i\}}
\bigg({\al_j\afbse_j-\al_i\afbse_i\atop\la-\dt-b\afbse_i}\bigg)
\bigg({\dt_i+\al_i\atop b,\,\dt_i,\,\al_i-b}\bigg)
A^{(\al)}\{\al_i\afbse_i+\dt,r\}\\
&=\bigg({\dt_i+\al_i\atop\al_i}\bigg)\sum_{\dt\in\afmbnn,\,\dt\leq\la\atop
0\leq b\leq\min\{\la_i-\dt_i,\al_i\}}
\bigg({\al_j\afbse_j-\al_i\afbse_i\atop\la-\dt-b\afbse_i}\bigg)
\bigg({\al_i\atop b}\bigg)
A^{(\al)}\{\al_i\afbse_i+\dt,r\}.
\end{split}
\]
The proposition is proved.
\end{proof}
Let
\begin{equation}\label{L+-}
\sL^+=\{(i,j)\mid 1\leq i\leq n,\ j\in\mbz,\,i<j\}\text{ and } \sL^-=\{(i,j)\mid 1\leq i\leq n,\ j\in\mbz,\,i>j\}.
\end{equation}
Furthermore, let $\sL=\sL^+\cup\sL^-$.
\begin{Coro}\label{tri1}
Assume $A\in\afThnpm$. Then the following triangular relation holds in $\afuglq:$
$$\prod_{(i,j)\in\sL}
(a_{i,j}\afE_{i,j})\{\bfl\}=A\{\bfl\}+\sum_{B\in\afThnpm,\,\dt\in\afmbnn
\atop\sg(B)<\sg(A)}f_{B,\dt,A}B\{\dt\}
$$
where $f_{B,\dt,A}\in\mbz$ and  the products are taken with respect to any fixed total order on $\sL$.
\end{Coro}
\begin{proof}
By \ref{fundamental formula} we have
\begin{equation}\label{la=0}
\begin{split}
(k\afE_{i,j})\{\bfl\}A\{\bfl\}&=\sum_{\al\in\La(\infty,k)}
\prod_{t\not=i\atop t\in\mbz}
\left({a_{i,t}+\al_t-\dt_{\bar j,\bar i}\al_{t+j-i}\atop\al_t}\right)A^{(\al)}\{\al_i\afbse_i\}
\end{split}
\end{equation}
where $A^{(\al)}$ is as in \ref{fundamental formula}.
Let $\bse_j=(\dt_{i,j})_{i\in\mbz}\in\La(\infty,1)$. If $\al\in\La(\infty,k)$ and $\al\not=k\bse_j$ then we have $\sg(A^{(\al)})<\sg(A)+k$. Consequently, for $k\geq 1$ and $i\not=j$, we have
\begin{equation}
(k\afE_{i,j})\{\bfl\}A\{\bfl\}
=\bigg({a_{i,j}+k\atop k}\bigg)
(A+k\afE_{i,j})\{\bfl\}+f,
\end{equation}
where $f$ is a $\mbz$-linear combination of $B\{\dt\}$ such that
$B\in\afThnpm$, $\dt\in\afmbnn$ and $\sg(B)<\sg(A)+k$.
Now, induction on $\sg(A)$ yields the result.
\end{proof}

As an application of \ref{fundamental formula}, we obtain the following multiplication formulas in the Ringel--Hall algebra $\Hall_\mbz\cong\afuglzp$ over $\mbz$.

\begin{Prop}\label{formula in Hall algebra}
Let $k\geq 1$ and $i,j\in\mbz$ with $i<j$. Assume $A\in\afThnp$. The following multiplication formulas hold in $\afuglzp:$
\[
\begin{split}
u^+_{k\afE_{i,j},1}u^+_{A,1}
&=\sum_{\al\in\La(\infty,k)\atop\al_t=0,\,\forall t\leq i}\prod_{t\not=i \atop
t\in\mbz}\left({a_{i,t}+\al_t-\dt_{\bar j,\bar i}\al_{t+j-i}\atop\al_t}\right)
u^+_{A^{(\al)},1}
=\bigg({a_{i,j}+k\atop k}\bigg)u^+_{A+k\afE_{i,j},1}+f
\end{split}
\]
where $A^{(\al)}$ is as in \ref{fundamental formula} and $f$ is a $\mbz$-linear combination of $u_{B,1}^+$ such that
$B\in\afThnp$ and $\sg(B)<\sg(A)+k$. A similar result holds for $\afuglzm$.
\end{Prop}
\begin{proof}
Assume
$A^{(\al)}\{\al_i\afbse_i\}\not=0$
for some $\al\in\La(\infty,k)$. Write $A^{(\al)}=(a_{s,t}^{(\al)})_{s,t\in\mbz}$. Then $a_{s,t}^{(\al)}\geq 0$ for all $s,t$. By \eqref{la=0}, it is enough to prove that $\al_t=0$ for $t\leq i$.
Since $A\in\afThnp$ and $j>i$ we have $a_{j,t}=0$ for $t\leq i$.
It follows that
$\dt_{\bar j,\bar i}\al_{t+i-j}-\al_t=a_{j,t}+\dt_{\bar j,\bar i}\al_{t+i-j}-\al_t=a_{j,t}^{(\al)}\geq 0$
for $t\leq i$. Thus, if $\bar j\not=\bar i$ then $\al_t=0$ for $t\leq i$.
Now we assume $j=i+mn$ for some $m>0$. Then $\al_{t-mn}\geq\al_t$ for $t\leq i$. This implies that if $t\leq i$, then there exists some $u\geq 1$ such that $0=\al_{t-umn}\geq\al_t$ and hence $\al_t=0$. The proposition is proved.
\end{proof}

Let
$\afThnm=\{A\in\afThn\mid a_{i,j}=0\text{ for }i\leq j\}.$
For $A\in\aftiThn$, write $A=A^++A^-+\diag(\la)$ with
$A^+\in\afThnp$, $A^-\in\afThnm$ and $\la\in\afmbzn$.
We fix a total order on $\sL^+,\sL^-$ and define for
$A\in\afThnpm$
\begin{equation}\label{MAla}
\begin{split}
\Eap&=\prod_{(i,j)\in\sL^+}u^+_{a_{i,j}\afE_{i,j},1}
=\prod_{(i,j)\in\sL^+}(a_{i,j}\afE_{i,j})\{\bfl\}\in\afuglqp, \\
\Faf&=\prod_{(i,j)\in\sL^-}u^-_{a_{i,j}\afE_{j,i},1}
=\prod_{(i,j)\in\sL^-}(a_{i,j}\afE_{i,j})\{\bfl\}\in\afuglqm.
\end{split}
\end{equation}

\begin{Coro}\label{tri for the postive part of affine gln}
Let $A\in\afThnp$. Then the following triangular relation holds in $\afuglzp:$
$$\Ea=u_{A,1}^++\sum_{B\in\afThnp
\atop\sg(B)<\sg(A)}f_{B,A}u_{B,1}^+
$$
where $f_{B,A}\in\mbz$. In particular,
the set $\{\Ea\mid A\in\afThnp\}$ forms a $\mbz$-basis for $\sU_\mbz^+(\afgl)$. A similar result holds for $\afuglzm$.
\end{Coro}

We are now ready to generalize \ref{tri1} to a little bit more general setting.  We need the following simple lemma.

\begin{Lem}\label{generalized commute formula}
Let $A_i\in\afThnpm$ ($1\leq i\leq m$) and $\la\in\afmbnn$.
Then we have
$$A_1\{\bfl\}\cdots
A_m\{\bfl\}\bigg({H\atop\la}\bigg)
-\bigg({H\atop\la}\bigg)A_1\{\bfl\}
\cdots A_m\{\bfl\}\in
\mpX^{\la} A_1\{\bfl\}
\cdots A_m\{\bfl\},$$
where $\mpX^{\la}=\spann_\mbz\{ \big({H\atop\dt}\big)\mid \dt\in\afmbnn,\,\dt<\la
\}$.
\end{Lem}
\begin{proof}
We argue by induction on $m$. If $m=1$ then by \ref{0mu Ala}  we have
\begin{equation}\label{A0 0la}
\begin{split}
 A_1\{\bfl\}\bigg({H\atop\la}\bigg)&=A_1\{\la\}+\sum_{\dt\in\afmbnn\atop\dt<\la}
\bigg({\co(A_1)\atop\la-\dt}\bigg)A_1\{\dt\}.
\end{split}
\end{equation}
This together with \eqref{0la A0} implies that
\begin{equation}\label{commute formula}
\begin{split}
 A_1\{\bfl\}\bigg({H\atop\la}\bigg)-\bigg({H\atop\la}\bigg)A_1\{\bfl\}&=\sum_{\dt\in\afmbnn\atop\dt<\la}
\bigg(\bigg({\co(A_1)\atop\la-\dt}\bigg)-\bigg({\ro(A_1)\atop\la-\dt}
\bigg)\bigg)A_1\{\dt\}\in\mpX^{\la}A_1\{\bfl\}
\end{split}
\end{equation}
Assume now that $m>1$.
By the inductive hypothesis we have
\[
\begin{split}
A_1\{\bfl\}A_2\{\bfl\}\cdots A_m\{\bfl\}\bigg({H\atop\la}\bigg)
-A_1\{\bfl\}\bigg({H\atop\la}\bigg)A_2\{\bfl\}\cdots A_m\{\bfl\}&\in A_1\{\bfl\}\mpX^{\la}
A_2\{\bfl\}\cdots A_m\{\bfl\}\\
&\han \mpX^{\la}A_1\{\bfl\}
A_2\{\bfl\}\cdots A_m\{\bfl\}.
\end{split}
\] Thus by \eqref{commute formula} we conclude that
$A_1\{\bfl\}A_2\{\bfl\}\cdots A_m\{\bfl\}\big({H\atop\la}\big)
-\big({H\atop\la}\big)A_1\{\bfl\}A_2\{\bfl\}\cdots A_m\{\bfl\}\in\mpX^{\la}A_1\{\bfl\}
A_2\{\bfl\}\cdots A_m\{\bfl\}$. The lemma is proved.
\end{proof}

\begin{Prop}\label{tri for affine gln}
For $A\in\afThnpm$ and $\la\in\afmbnn$ we have the following triangular relation in $\afuglq:$
\[
\Eap \bigg({H\atop\la}\bigg)\Faf=
A\{\la\}+\sum_{\dt\in\afmbnn\atop\dt<\la}g_{A,\la}^\dt A\{\dt\}
+\sum_{B\in\afThnpm,\,\dt\in\afmbnn\atop\sg(B)<\sg(A)}
g_{A,\la}^{B,\dt}B\{\dt\},
\]
where $g_{A,\la}^\dt,g_{A,\la}^{B,\dt}\in\mbz$.
\end{Prop}
\begin{proof}
Let $\mpX_A^{\la}=\spann_{\mbz}\{A\{\dt\}\mid \dt\in\afmbnn,\,\dt<\la\}$ and $\mpX^{A}=\spann_\mbz
\{B\{\dt\}\mid B\in\afThnpm,\,\dt\in\afmbnn,\,\sg(B)<\sg(A)\}$.
Combining \ref{tri1} and \ref{generalized commute formula}  implies that
\[
\begin{split}
\Eap \bigg({H\atop\la}\bigg)\Faf&=\bigg(\bigg({H\atop\la}\bigg)+f\bigg)
\Eap \Faf =\bigg(\bigg({H\atop\la}\bigg)+f\bigg)(A\{\bfl\}+g)
\end{split}
\]
where $f\in\mpX^\la$ and $g\in\mpX^A$. From \ref{0mu Ala} and \eqref{0la A0}, we see that $f(A\{\bfl\}+g)\in\mpX_A^{\la}+\mpX^{A}$, $\big({H\atop\la}\big)g\in\mpX^A$ and $\big({H\atop\la}\big)A\{\bfl\}-A\{\la\}\in\mpX_A^\la$. Consequently, we have
$\Eap \big({H\atop\la}\big)\Faf-A\{\la\}\in\mpX_A^{\la}+\mpX^{A}$. The proposition is proved.
\end{proof}

\begin{Thm}\label{basis2 for afuglz}
Each of the following sets forms a $\mbz$-basis for $\sU_\mbz(\afgl):$
\begin{itemize}
\item[(1)]
$\fM=\{\Eap \big({H\atop\la}\big)\Faf\mid A\in\afThnpm,\,\la\in\afmbnn\};$
\item[(2)]
$\fB=\{A\{\la\}\mid A\in\afThnpm,\,\la\in\afmbnn\};$
\item[(3)]
$\fB'=\{A\{\bfl\}\big({H\atop\la}\big)\mid A\in\afThnpm,\,\la\in\afmbnn\}.$
\end{itemize}
In particular, $\sU_\mbz(\afgl)$ is generated by the elements
$u^+_{k\afE_{i,j},1}$, $u^-_{k\afE_{i,j},1}$ and $\big({H_i\atop k}\big)$ for $k\in\mbn$, $1\leq i\leq n$, $j\in\mbz$ and $i<j$.
\end{Thm}
\begin{proof}
Combining \ref{basis for afuglz} and
\ref{tri for the postive part of affine gln} implies that the set
$\fM$ forms a $\mbz$-basis for $\sU_\mbz(\afgl)$. The remaining assertion follows from
\eqref{A0 0la} and \ref{tri for affine gln}.
\end{proof}

By \eqref{th+uA} and \eqref{la=0} we have the following result.

\begin{Lem}\label{divided power}
Let $k\in\mbn$ and $i,j\in\mbz$ with $\bar i\not=\bar j$. Then
\begin{equation*}
(\afE_{i,j})^{(k)}
=(k\afE_{i,j})\{\bfl\}=
\begin{cases}
u^+_{k\afE_{i,j}}&\text{if $i<j$},\\
u^-_{k\afE_{j,i}}&\text{if $i>j$},
\end{cases}
\end{equation*}
where $(\afE_{i,j})^{(k)}=\frac{(\afE_{i,j})^k}{k!}$.
\end{Lem}

\begin{Rem}
$(1)$ From \ref{basis2 for afuglz} and \ref{divided power}, we see that $\afuglz$ is generated by the elements
$(\afE_{i,j})^{(k)}$, $u^+_{k\afE_{i,i+mn},1}$, $u^-_{k\afE_{i,i+mn},1}$ and $\big({H_i
\atop k}\big)$ for $k\in\mbn$, $1\leq i\leq n$, $j\in\mbz$ with $\bar i\not=\bar j$, $m>0$.
It should be noted that for $k\geq 2$, $\frac{(\afE_{i,i+mn})^k}{k!}\not\in\afuglz$ in general.

$(2)$ Let $\sU_\mbz(\frak{gl}_n)$ be the $\mbz$-subalgebra of $\sU_\mbz(\afgl)$ generated by the elements
$(\afE_{i,j})^{(k)}$ and $\big({H_i
\atop k}\big)$ for $k\in\mbn$, $1\leq i,j\leq n$ with $i\not=j$. Then the set
$\{\Eap \big({H\atop\la}\big)\Faf\mid A\in\Thnpm,\,\la\in\afmbnn\}$
is the PBW-basis of $\sU_\mbz(\frak{gl}_n)$ constructed by Kostant \cite{Ko},
where $\Thnpm$ is the subset of $\afThnpm$ consisting of all  $A\in\afThnpm$ such that $a_{i,j}=0$ for $1\leq i\leq n$ and  $j\not\in[1,n]$. So we may regard
the set $\fM$ as the PBW-basis of $\sU_\mbz(\afgl)$.

\end{Rem}

\section{The Garland integral form of $\afuglq$}

Let $\mathscr S=\mbq[X_i\mid i\geq 1]$ be the polynomial algebra over $\mbq$ in infinitely many indeterminates.
Following \cite{Ga}, we define a derivation
$D_+:\mathscr S\ra\mathscr S$ by setting
$$D_+(X_i)=iX_{i+1}$$
for $i\geq 1$.
Let $\La_0=1$. For $k\geq 1$ let $$\La_k=\frac{(D_++L_{X_1})^{k-1}}{k!}(X_1).$$
For $x\in\mathscr S$ define $L_x:\mathscr S\ra\mathscr S$ by $L_x(y)=xy$, $y\in\mathscr S$.
\begin{Lem}[{\cite[5.2]{Ga}}]\label{Galand lemma}
For $k\geq 1$ we have
$$\frac{(L_{X_1}+D_+)^{k}}{k!}=\sum_{0\leq s\leq k}L_{\La_{s}}
\frac{D_+^{k-s}}{(k-s)!}.$$
\end{Lem}
\begin{Coro}\label{Lak}
For $k\geq 1$ we have
$$\La_k=\frac{1}{k}\sum_{0\leq s\leq k-1}\La_sX_{k-s}.$$
\end{Coro}
\begin{proof}
Clearly we have $D_+^s(X_1)=(s!)X_{s+1}$ for $s\geq 0$.
Thus by \ref{Galand lemma} we have
$$\La_k=\frac{1}{k}\sum_{0\leq s\leq k-1}L_{\La_{s}}
\frac{D_+^{k-1-s}}{(k-1-s)!}(X_1)=\frac{1}{k}\sum_{0\leq s\leq k-1}
\La_{s}X_{k-s}$$
as required.
\end{proof}

For $i,l\in\mbz$ with $l\not=0$ we define an algebra homomorphism
\begin{equation}\label{Psiil}
\Psi_{i,l}:\mathscr S\ra\afuglq
\end{equation}
by $\Psi_{i,l}(X_m)=\afE_{i,i+mln}$.
Let $\Uz$ be the $\mbz$-subalgebra of $\afuglq$ generated by
the elements
$(\afE_{i,j})^{(k)}$, $\Psi_{i,l}(\La_k)$ and $\big({H_i
\atop k}\big)$ for $k\in\mbn$, $1\leq i\leq n$, $j,l\in\mbz$ with $\bar i\not=\bar j$, $l\not=0$. We will prove in \ref{identification} that
$\Uz$ coincides with $\afuglz$.

For $i\not=j\in\mbz$ and $k\in\mbn$ let
$$\msX_{i,j}^{(k)}=
\begin{cases}
\frac{(\afE_{i,j})^k}{k!}&\text{if $\bar i\not=\bar j$}\\
\Psi_{i,l}(\La_k)&\text{if $j=i+ln$ for some $l\not=0$}.
\end{cases}
$$
We fix a total order on $\sL^+,\sL^-$ and define for $A\in\afThnpm$
\begin{equation}\label{MAla}
\begin{split}
\msEap&=\prod_{(i,j)\in\sL^+}\msX_{i,j}^{(a_{i,j})}, \qquad
\msFaf=\prod_{(i,j)\in\sL^-}\msX_{i,j}^{(a_{i,j})}.
\end{split}
\end{equation}
\begin{Thm}[{\cite[Th. 5.8]{Ga}}]\label{Galand theorem}
The set $\{\msEap\big({H\atop\la}\big)\msFaf
\mid A\in\afThnpm,\,\la\in\afmbnn\}$ forms
a $\mbz$-basis of $\Uz$.
\end{Thm}

For $k\in\mbn$ let $\La^+(k)=\{\la=(\la_s)_{s\geq 1}\mid\la_s\in\mbn\   \forall s,\,\sum_{s\geq 1}\la_s=k\}$ and
$\ti\La^+(k)=\{\bfb=(b_s)_{s\geq 1}\mid b_s\in\mbn\   \forall s,\,\sum_{s\geq 1}sb_s=k\}$. Then
there is a bijective map
$$\vartheta_k:\La^+(k)\ra\ti\La^+(k)$$
sending $\la$ to $\bfb$, where
$b_s=|\{j\geq 1\mid \la_j=s\}|.$

For $i,l\in\mbz$ and $m\in\mbn$ with $l\not=0$ let
$$A_{m}^{(i,l)}=\begin{cases}
\afE_{i,i+mln}&\text{if $m\geq 1$}\\
0&\text{if $m=0$}.
\end{cases}$$
For $\la\in\La^+(k)$
let $$A_\la^{(i,l)}=\sum_{s\geq 1}A^{(i,l)}_{\la_s}.$$
Furthermore, for $\bfb\in\ti\La^+(k)$ let
$$\ti A_\bfb^{(i,l)}=\sum_{s\geq 1}b_sA_s^{(i,l)}.$$
Then we have
$\ti A_\bfb^{(i,l)}=A_{\vartheta_k^{-1}(\bfb)}^{(i,l)}$
for $\bfb\in\ti\La^+(k)$.

For $s\geq 1$ let $\bse_s=(\dt_{s,j})_{j\geq 1}$.
By \cite[Cor. 6.3.5]{DDF} we have the following result.
\begin{Lem}\label{lem-formulas for integral form}
Let $k\in\mbn$, $m\geq 1$ and $i,l\in\mbz$ with $l\not= 0$. Then
$$\afE_{i,i+mln}[\bfl]\ti A_\bfb^{(i,l)}[\bfl]=\sum_{s\geq 0}(b_{s+m}+1)
\ti A_{\bfb+\bse_{m+s}-\bse_s}^{(i,l)}[\bfl]$$
for $\bfb\in\ti\La^+(k)$, where $\bse_0=0$.
\end{Lem}

\begin{Prop}\label{key for integral form}
For $k\in\mbn$ and $i,l\in\mbz$ with $l\not=0$
we have
$$\Psi_{i,l}(\La_k)=\sum_{\la\in\La^+(k)}A_\la^{(i,l)}[\bfl]
=\begin{cases}
\sum_{\la\in\La^+(k)}u^+_{A_\la^{(i,l)},1}&\quad\text{if $l>0$,}\\
\sum_{\la\in\La^+(k)}u^-_{{}^t\!(A_\la^{(i,l)}),1}&\quad\text{if $l<0$,}
\end{cases}$$
where ${}^t\!(A_\la^{(i,l)})$ is the transpose of
$A_\la^{(i,l)}$.
In particular, we have
$$\Psi_{i,l}(\La_k)=
\begin{cases}
u^+_{kE_{i,i+ln},1}+f_1&\text{if $l>0$},\\
u^-_{kE_{i+ln,i},1}+f_2&\text{if $l<0$},
\end{cases}
$$
where $f_1$ stands for a $\mbz$-linear combination of $u^+_{B,1}$ with $B\in\afThnp$ such that $\sg(B)<k$ and
$f_2$ stands for a $\mbz$-linear combination of $u^-_{B,1}$ with $B\in\afThnp$ such that $\sg(B)<k$.
\end{Prop}
\begin{proof}
We argue by induction on $k$. The case where $k=0,1$ is trivial.
Assume now that $k\geq 2$.
By the induction hypothesis and \ref{Lak}  we have
$$\Psi_{i,l}(\La_k)=\frac{1}{k}\sum_{0\leq s\leq k-1\atop\bfb\in\ti\La^+(s)}
\afE_{i,i+(k-s)ln}[\bfl]\ti A^{(i,l)}_\bfb[\bfl].$$
It follows from
\ref{lem-formulas for integral form} that
$$
\Psi_{i,l}(\La_k)=\frac{1}{k}\sum_{0\leq s\leq k-1\atop\bfb\in\ti\La^+(s),\,j\geq 0}
(b_{j+k-s}+1)\ti A_{\bfb+\bse_{j+k-s}-\bse_j}[\bfl]
=\frac{1}{k}\sum_{\bfc\in\ti\La^+(k)}
g_\bfc\ti A_\bfc[\bfl],$$
where $$g_\bfc=\sum_{0\leq s\leq k-1,\,\bfb\in\ti\La^+(s),\,j\geq 0
\atop\bfc=\bfb+\bse_{j+k-s}-\bse_j}(b_{j+k-s}+1)=\sum_{0\leq s\leq k-1\atop
j\geq 0}c_{j+k-s}.$$
If $\bfc\in\ti\La^+(k)$ then $c_u=0$ for $u>k$.
This implies that
$$g_\bfc=\sum_{u\geq 1,\,0\leq s\leq k-1\atop u-k+s\geq 0}c_u
=\sum_{1\leq u\leq k}\bigg(\sum_{k-u\leq s\leq k-1}c_u\bigg)
=\sum_{1\leq u\leq k}uc_u=k$$
for $\bfc\in\ti\La^+(k)$. Thus
we have
$\Psi_{i,l}(\La_k)=\sum_{\bfc\in\ti\La^+(k)}\ti A_\bfc[\bfl]$. The proposition is proved.
\end{proof}

By \cite[Cor.~4.1.1]{Peng} (see also \cite[Lem. 6.1.3]{DDF}) we have the following result.
\begin{Lem}\label{A1A2 cdots At}
Let
$A_s=(a_{i,j}^{(s)})\in\afThnp$ ($1\leq s\leq t$).
Then $$u_{A_1,1}^+\cdots u_{A_t,1}^+=
\prod_{1\leq i\leq n\atop i<j,\,j\in\mbz}\frac{a_{i,j}!}
{a_{i,j}^{(1)}!a_{i,j}^{(2)}!\cdots a_{i,j}^{(t)}!}u_{A,1}^++f
$$
where $A=\sum_{1\leq s\leq t}A_s$ and $f$ is a $\mbz$-linear combination of $u^+_{B,1}$ with $B\in\afThnp$
and $\sg(B)<\sg(A)$.
\end{Lem}

\begin{Lem}\label{relation between Galand integral form and Ringel Hall algebras}
$(1)$
For $A\in\afThnp$ we have
$$\msEa=u_{A,1}^++\sum_{B\in\afThnp\atop\sg(B)<\sg(A)}g_{B,A}u_{B,1}^+$$
where $g_{B,A}\in\mbz$. In particular, the set $\{\msEa\mid
A\in\afThnp\}$ forms a $\mbz$-basis of $\afuglzp$.

$(2)$ For $A\in\afThnm$ we have
$$\msFa=u_{\tA,1}^-+\sum_{B\in\afThnp\atop\sg(B)<\sg(A)}h_{B,A}u_{B,1}^-$$
where $h_{B,A}\in\mbz$. In particular, the set $\{\msFa\mid
A\in\afThnm\}$ forms a $\mbz$-basis of $\afuglzm$.
\end{Lem}
\begin{proof}
Assume $A\in\afThnp$. For $k\in\mbn$ let $\mpY_k=\spann_\mbz\{u^+_{B,1}
\mid B\in\afThnp,\,\sg(B)<k\}$.
By \ref{key for integral form}, we have $$\msX_{i,i+ln}^{(a_{i,i+ln})}
=\Psi_{i,l}(\La_{a_{i,i+ln}})=u^+_{a_{i,i+ln}\afE_{i,i+ln},1}+f_{i,l}$$
for $1\leq i\leq n$ and
$l\geq 1$, where $f_{i,l}\in\mpY_{a_{i,i+ln}}$.
It follows from \ref{A1A2 cdots At} that
$\msEa=E^{(A)}+f_1$
where $f_1\in\mpY_{\sg(A)}$. By
\ref{tri for the postive part of affine gln} we have
$\Ea=u_{A,1}^++f_2$, where $f_2\in\mpY_{\sg(A)}$. Consequently, we have
$\msEa=u_{A,1}^++f$, where $f=f_1+f_2\in\mpY_{\sg(A)}$. The statement (1) is proved. The statement (2) can be proved similarly.
\end{proof}

Combining \ref{Galand theorem} and \ref{relation between Galand integral form and Ringel Hall algebras}, we obtain the following result.
\begin{Thm}\label{identification}
The set $\fG=\{\msEap\big({H\atop\la}\big)\msFaf
\mid A\in\afThnpm,\,\la\in\afmbnn\}$ forms
a $\mbz$-basis of $\afuglz$. In particular we have
$\afuglz=\Uz$.
\end{Thm}

\section{The hyperalgebra $\Unk$ and its subalgebras $\Unkh$}

Let $\field$ be a commutative ring with unity.
Assume $p=\text{char}\field>0$.
Let $$\Unk=\Un\ot_\mbz\field.$$ We shall denote the images of $(\afE_{i,j})^{(k)}$, $\big({H_i\atop k}\big)$, $A\{\la\}$, etc. in $\Unk$ by the same letters.
We will refer to $\Unk$ as the hyperalgebra of $\afgl$.

For $h\geq 1$ let $\Unkh$ be the $\field$-subalgebra of $\Unk$ generated by the elements $u^+_{k\afE_{i,j},1}$, $u^-_{k\afE_{i,j},1}$ and $\big({H_i
\atop k}\big)$ for $0\leq k< p^h$, $1\leq i\leq n$, $j\in\mbz$ with $i<j$.
By \ref{divided power}, we know that $\Unkh$ is generated by the elements
$(\afE_{i,j})^{(k)}$, $u^+_{k\afE_{i,i+ln},1}$, $u^-_{k\afE_{i,i+ln},1}$ and $\big({H_i
\atop k}\big)$ for $0\leq k<p^h$, $1\leq i\leq n$, $j\in\mbz$ with $\bar i\not=\bar j$, $l>0$. Clearly we have $\Unko\han\Unkt\han\cdots\han\Unk$ and
\begin{equation}\label{bin Unkh}
\Unk=\bin_{h\geq 1}\Unkh.
\end{equation}


\begin{Rem}
Let $\ttu(\frak{gl}_n)_h$ be the subalgebra of $\Unkh$ generated by the elements $u^+_{k\afE_{i,j},1}$, $u^-_{k\afE_{i,j},1}$ and $\big({H_i
\atop k}\big)$ for $0\leq k< p^h$, $1\leq i,j\leq n$ with $i<j$.
From \ref{divided power} we see that $\ttu(\frak{gl}_n)_h$ is generated by the elements $(\afE_{i,j})^{(k)}$ and $\big({H_i\atop k}\big)$ for $0\leq k< p^h$, $1\leq i,j\leq n$ with $i\not=j$. The algebra $\Unkh$ is the affine analogue of $\ttu(\frak{gl}_n)_h$.
The algebra $\ttu(\frak{gl}_n)_h$ is introduced by Humphreys in \cite{Hum} and it plays an important role in the modular representation theory of the algebraic group of type $A$.
\end{Rem}

We will construct several $\field$-bases for $\Unkh$ in \ref{basis of Unkh}. Before proving \ref{basis of Unkh}, we need some preparation.
The following simple lemma will be needed later(cf. \cite[Lem.~3.2]{Fu15b}).
\begin{Lem}\label{identity}
$(1)$ The following identity holds in $\field:$
$\big({t+p^{h}\atop s}\big)=\big({t\atop s}\big)$
for $t\in\mbz$ and $0\leq s<p^{h}$.

$(2)$ Assume $0\leq a,b<p^{h}$ and $a+b\geq p^{h}$. Then $\big({a+b\atop a}\big)=0$ in $\field$.
\end{Lem}

Let $\Unkhz$ be the $\field$-subalgebra of $\Unkh$ generated by $\big({H_i
\atop k}\big)$ for $1\leq i\leq n$ and $0\leq k<p^h$.
For $h\geq 1$ let
$$\afmbnnh=\{\la\in\afmbnn\mid 0\leq\la_i<p^{h},\,\forall i\}.$$

\begin{Lem}\label{basis B0}
The set $\fM_h^0=\{\big({H
\atop \la}\big)
\mid \la\in\afmbnnh\}$ forms
a $\field$-basis for $\Unkhz$.
\end{Lem}
\begin{proof}
Let $V=\spann_\field\fM_h^0$.
From \cite[2.3(g8)]{Lu90},
we see that
$$\lebr{H_i\atop t'}\ribr
\lebr{H_i\atop t}\ribr=\lebr{t+t'\atop t}\ribr\lebr{H_i\atop t+t'}\ribr-\sum_{0<j\leq t'}(-1)^j\lebr{t+j-1\atop j}\ribr
\lebr{H_i\atop t'-j}\ribr
\lebr{H_i\atop t}\ribr
$$
for $0\leq t,t'<p^{h}$.
By \ref{identity} we have
$\big({t+t'\atop t}\big)\big({H_i\atop t+t'}\big)=0$
for $0\leq t,t'<p^{h}$ with $t+t'\geq p^{h}$. Consequently, by induction on $t'$ we see that
$\bblr{H_i\atop t'}\bbrr
\bblr{H_i\atop t}\bbrr\in V$ for $0\leq t,t'<p^{h}$. It follows that
$\Unkhz=V$.  Furthermore,  by \ref{basis for afuglz} the set $\fM_h^0$ is linearly independent. The lemma is proved.
\end{proof}

For $h\geq 1$ let $$\afThnpmh=\{A\in\afThnpm\mid 0\leq a_{s,t}<p^{h},\,\forall s\not=t\}.$$
Let $$V_h=\spann_\field\{A\{\la\}\mid A\in\afThnpmh,\,\la\in\afmbnnh\}\han\Unk.$$
\begin{Lem}\label{property of Vh}
Let $0\leq k<p^h$, $1\leq i\leq n$, $j\in\mbz$ and $i\not=j$.
Assume $A\in\afThnpmh$ and $\la,\mu\in\afmbnnh$. Then we have
$(k\afE_{i,j})\{\bfl\}A\{\la\}\in V_h$ and $\big({H\atop\mu}\big)A\{\la\}\in V_h$.
\end{Lem}
\begin{proof}
We use \ref{fundamental formula} to prove
$(k\afE_{i,j})\{\bfl\}A\{\la\}\in V_h$.
If  $A^{(\al)}\in\afThnpm\backslash
\afThnpmh$ for some $\al\in\La(\infty,k)$ then $a_{i,t}+\al_t-\dt_{\bar i,\bar j}\al_{t+j-i}\geq p^h$ for some $t\not=i$.
Since $A\in\afThnpmh$ we have $0\leq a_{i,t}+\al_t-\dt_{\bar i,\bar j}\al_{t+j-i}-p^h\leq \al_t+a_{i,t}-p^h<\al_t$. Furthermore we have
$0\leq\al_t\leq k<p^h$. It follows from
\ref{identity} that
$$\lebr{a_{i,t}+\al_t-\dt_{\bar i,\bar j}\al_{t+j-i}\atop\al_t}\ribr
=\lebr{a_{i,t}+\al_t-\dt_{\bar i,\bar j}\al_{t+j-i}-p^h\atop\al_t}\ribr=0$$
in $\field$.
In addition, if $\al_i\afbse_i+\dt\not\in\afmbnnh$ for some $\al\in\La(\infty,k)$ and $\dt\in\afmbnn$ with $\dt\leq\la$, then $\al_i+\dt_i\geq p^h$. Since $k<p^h$ and $\la\in\afmbnnh$, we have $\al_i\leq k<p^h$ and $\dt_i\leq\la_i<p^h$. This together with \ref{identity} implies that
$\big({\al_i+\dt_i\atop\al_i}\big)=0$ in $\field$. Consequently,
$(k\afE_{i,j})\{\bfl\}A\{\la\}\in V_h$.

If $\la+\dt\not\in\afmbnnh$ for some $\dt\in\afmbnn$ with $\dt\leq\mu$ then $\la_t+\dt_t\geq p^h$ for some $t$. Since $\la,\mu\in\afmbnnh$ we have $\la_t<p^h$ and $\dt_t\leq\mu_t<p^h$. It follows from \ref{identity} that
$\big({\la_t+\dt_t\atop\la_t}\big)=0$ in $\field$. Hence, by
\ref{0mu Ala} we conclude that $\big({H\atop\mu}\big)A\{\la\}\in V_h$.
\end{proof}

For $h\geq 1$ let $\Unkhp$ (resp. $\Unkhm$) be the $\field$-subalgebra of $\Unkh$ generated by the elements $u^+_{k\afE_{i,j},1}$ (resp. $u^-_{k\afE_{i,j},1}$) for $0\leq k< p^h$, $1\leq i\leq n$, $j\in\mbz$ with $i<j$. Let $\afThnph=\{A\in\afThnp\mid 0\leq a_{s,t}<p^{h},\,\forall s\not=t\}$ and
$\afThnmh=\{A\in\afThnm\mid 0\leq a_{s,t}<p^{h},\,\forall s\not=t\}$.

\begin{Lem}\label{basis of Unkhp}
Each of the following set forms a $\field$-basis for $\Unkhp:$
\begin{itemize}
\item[(1)]
$\fM_h^+:=\{\Ea\mid A\in\afThnph\};$
\item[(2)]
$\fC_h^+:=\{u^+_{A,1}\mid A\in\afThnph\};$
\item[(3)]
$\fG_h^+:=\{\msEa\mid A\in\afThnph\}$.
\end{itemize}
A similar result holds for $\Unkhm$.
\end{Lem}
\begin{proof}
Let $V_h^+=\spann_\field\fC_h^+$.
Combining \ref{formula in Hall algebra} and \ref{property of Vh}, we conclude that $\Unkhp\han\Unkhp V_h^+\han V_h^+.$
This together with \ref{tri for the postive part of affine gln} implies that for $A\in\afThnph$ we have
$$\Ea=u_{A,1}^++f$$
where $f$ is a $\field$-linear combination of $u_{B,1}^+$ with $B\in\afThnph$ such that $\sg(B)<\sg(A)$. It follows that $V_h^+=\spann_\field\fM_h^+=\Unkhp$.
This together with \ref{key for integral form} implies that $\msE^{(A)}\in V_h^+$ for $A\in\afThnph$. It follows from \ref{relation between Galand integral form and Ringel Hall algebras} that  for $A\in\afThnph$ we have
$$\msEa=u_{A,1}^++g,$$
where $g$ stands for a $\field$-linear combination of $u_{B,1}^+$ with $B\in\afThnph$ such that $\sg(B)<\sg(A)$. So $V_h^+=\spann_\field\fG_h^+$.
Now the assertion follows from \ref{basis2 for afuglz} and \ref{identification}.
\end{proof}

\begin{Prop}\label{basis of Unkh}
Each of the following set forms a $\field$-basis for $\Unkh:$
\begin{itemize}
\item[(1)]
$\fM_h:=\{\Eap \big({H\atop\la}\big)\Faf\mid A\in\afThnpmh,\,\la\in\afmbnnh\};$
\item[(2)]
$\fC_h:=\{u^+_{A^+,1}\big({H\atop\la}\big)
u^-_{{}^t\!(A^-),1}\mid A\in\afThnpmh,\,\la\in\afmbnnh\};$
\item[(3)]
$\fG_h:=\{\msEap\big({H\atop\la}\big)\msFaf\mid A\in\afThnpmh,\,\la\in\afmbnnh\};$
\item[(4)]
$\fB_h:=\{A\{\la\}\mid A\in\afThnpmh,\,\la\in\afmbnnh\}.$
\end{itemize}
\end{Prop}
\begin{proof}
By \ref{basis for afuglz}, \ref{basis2 for afuglz} and \ref{identification}, we know that
each of the sets (1)-(4) is linear independent. Furthermore,
by \ref{basis of Unkhp} we have
$\spann_\field\fM_h=\spann_\field\fC_h=\spann_\field\fG_h$. Thus
it is enough to prove that
$\Unkh=V_h=\spann_\field\fM_h$, where
$V_h=\spann_\field\fB_h$.
Combining \eqref{th+uA} with \ref{property of Vh}
implies that
\begin{equation}\label{Unkhp Vh}
\Unkh\han\Unkh V_h\han V_h.
\end{equation}
Furthermore, by \ref{tri for affine gln}, for $A\in\afThnpmh$ and $\la\in\afmbnnh$ we have
$$\Eap\bigg({H\atop\la}\bigg)\Faf=A\{\la\}+f+g,$$
where $f$ is a $\field$-linear combination of
$A\{\dt\}$ with $\dt\in\afmbnn$ such that $\dt<\la$, and $g$ is a $\field$-linear combination of
$B\{\dt\}$ with $B\in\afThnpm$ and $\dt\in\afmbnn$ such that $\sg(B)<\sg(A)$.
From \eqref{Unkhp Vh} we see that $f$ must be a linear combination of
$A\{\dt\}$ with $\dt\in\afmbnnh$ such that $\dt<\la$, and $g$ must be a $\field$-linear combination of
$B\{\dt\}$ with $B\in\afThnpmh$ and $\dt\in\afmbnnh$ such that $\sg(B)<\sg(A)$.
It follows that $V_h=\spann_\field\fM_h\han\Unkh$. The proposition is proved.
\end{proof}
\begin{Rem}
From \ref{basis of Unkh}(3) we see that $\Unkh$
is generated by the elements
$(\afE_{i,j})^{(k)}$, $\Psi_{i,l}(\La_k)$ and $\big({H_i
\atop k}\big)$ for $0\leq k<p^h$, $1\leq i\leq n$, $j\in\mbz$ with $\bar i\not=\bar j$, $l\not=0$, where $\Psi_{i,l}$ is given in \eqref{Psiil}.
\end{Rem}

Combining \ref{basis B0}, \ref{basis of Unkhp} and \ref{basis of Unkh} yields the following result.

\begin{Coro}
We have $\Unkh=\Unkhp\Unkhz\Unkhm\cong\Unkhp\ot\Unkhz\ot\Unkhm$.
\end{Coro}

Let
$\afSrk=\afSrmbz\ot \field.$ Again,
by abuse of notation, we shall denote the images of $[A]_1,A\{\la,r\}$, etc. in $\afSrk$ by the same letters.
In \ref{realization}, the algebra $\afuglq$ over $\mbq$ is embedded into $\prod_{r\geq 0}\afSrmbq$. We will prove in \ref{injective} that the hyperalgebra
$\afuglk$ and $\Unkh$ can be realized as subalgebras of $\prod_{r\geq 0}\afSrk$.

Recall the map $\xi_r:\afuglq\ra\afSrmbq$ defined in \eqref{xir}.
It is proved in \cite[6.7]{Fu13} that $\xi_r(\afuglz)=\afSrmbz$.
Thus, the restriction of $\xi_r$ to $\afuglz$ gives a surjective $\mbz$-algebra homomorphism $\xi_r:\afuglz\tra\afSrmbz.$
By tensoring  with $\field$,
we get a surjective algebra homomorphism
\begin{equation}\label{xrk}
\xrk:=\xi_r\ot id:\Unk\tra\afSrk.
\end{equation}
The maps $\xi_{r,\field}$ induce an algebra homomorphism
\begin{equation}\label{zeta}
\xi_\field=\prod_{r\geq 0}\xi_{r,\field}:\Unk\ra\prod_{r\geq 0}\afSrk.
\end{equation}

\begin{Thm}\label{injective}
The algebra homomorphism $\xi_\field:\Unk\ra\prod_{r\geq 0}\afSrk$ is injective. Furthermore we have $\xi_\field(\Unk)=\spann_\field\{(A\{\la,r\})_{r\geq 0}\mid A\in\afThnpm,\,\la\in\afmbnn\}$ and
$\xi_\field(\Unkh)=\spann_\field\{(A\{\la,r\})_{r\geq 0}\mid A\in\afThnpmh,\,\la\in\afmbnnh\}$.
\end{Thm}
\begin{proof}
Clearly, we have
\begin{equation}\label{xik(Ala)}
\xi_\field(A\{\la\})=(A\{\la,r\})_{r\geq 0}
\end{equation}
for $A\in\afThnpm$ and $\la\in\afmbnn$.
From \ref{basis2 for afuglz} we see that the set
$\{A\{\la\}\mid A\in\afThnpm,\,\la\in\afmbnn\}$ forms a $\field$-basis for $\Unk$.
Thus it is enough to prove that the set
$\{(A\{\la,r\})_{r\geq 0}\mid A\in\afThnpm,\,\la\in\afmbnn\}$ is $\field$-linear independent.

Suppose $$\sum_{A\in\afThnpm\atop\la\in\afmbnn}f_{A,\la}  (A\{\la,r\})_{r\geq 0}=0,$$
where $f_{A,\la}\in\field$. Then for any $r\geq 0$ we have
$$0=
\sum_{A\in\afThnpm\atop\la\in\afmbnn}f_{A,\la} A\br{\la,r}
=\sum_{A\in\afThnpm\atop\mu\in\afLa(n,r-\sg(A))}
\left(
\sum_{\la\in\afmbnn}f_{A,\la}
\bigg({\mu\atop\la}\bigg)\right)[A+\diag(\mu)]_1$$
in $\afSrk$.
This implies that $\sum_{\la\in\afmbnn}f_{A,\la}
\left({\mu\atop\la}\right)=0,$
for $A\in\afThnpm$, $\mu\in\afLa(n,r-\sg(A))$ with $r\geq\sg(A)$.
For $A\in\afThnpm$, we have $\afmbnn=\bin_{r\geq\sg(A)}\afLa(n,r-\sg(A))$. It follows that
$$\sum_{\la\in\afmbnn}f_{A,\la}
\left({\mu\atop\la}\right)=0$$ for $A\in\afThnpm$ and $\mu\in\afmbnn$.
For $A\in\afThnpm$ let $X_A=\{\la\in\afmbnn\mid f_{A,\la}\not=0\}$. If $X_A\not=\varnothing$ for some $A\in\afThnpm$, let $\nu$ be the minimal element in $X_A$ with respect to the order relation $\leq$ on $\afmbzn$ defined in \eqref{order on afmbzn}. Since $\nu$ is minimal, we have
$f_{A,\la}=0$ for $\la\in\afmbnn$ with $\nu>\la$. This implies that
$$f_{A,\nu}=f_{A,\nu}+\sum_{\la\in\afmbnn\atop\nu>\la}f_{A,\la}
\left({\nu\atop\la}\right)=\sum_{\la\in\afmbnn}f_{A,\la}
\left({\nu\atop\la}\right)=0.$$
This is a contradiction. Consequently, $X_{A}=\varnothing$ for any $A\in\afThnpm$. The theorem is proved.
\end{proof}

\section{A realization of $\Unkh$}
In this section we will first construct the algebra $\msKnhq$. The algebra
$\msKnhq$ is the affine analogue of the algebra $\ms K'$ constructed by
Beilinson--Lusztig--MacPherson in \cite[\S6]{BLM}. We will prove in \ref{vi} that
$\msKnhq$ is a realization of $\Unkh$.

The stabilization property of the multiplication of affine quantum Schur algebras was established in \cite[Prop.~6.3]{DF14}. Using this property,  a certain algebra $\afKn$ over $\sZ$ (without
1), with basis $\{[A]\mid A\in\aftiThn\}$ was constructed in
\cite[\S6]{DF14}.

Let $\afKnmbz=\afKn\ot\mbz$, where $\mbz$ is regarded as a $\sZ$-module by specializing $v$ to $1$. By \cite[Prop.~6.2]{Fu12} and \cite[Lem.~6.1]{DF14}, we know that $\afKnmbz$ is isomorphic to the algebra $\sK_\mbz(n)$ defined in \cite[\S6]{Fu12}.
For $A\in\aftiThn$ let
$$[A]_1=[A]\ot 1\in\afKnmbz.$$
Then the set $\{[A]_1\mid A\in\aftiThn\}$ forms a $\mbz$-basis for $\afKnmbz$.
Note that if $A,B\in\tiThn$ is such that $\co(B)\not=\ro(A)$ then $[B]_1\cdot[A]_1=0$ in $\afKnmbz$.
By \ref{formula}, \ref{tri1} and the definition of $\afKnmbz$, we have the following result.

\begin{Lem}\label{formula in Knmbz}
$(1)$ Let $k\geq 1$ and $i,j\in\mbz$ with $i\not=j$. Assume $A\in\aftiThn$ and $\la=\ro(A)$. Then the following formulas hold in $\afKnmbz:$
\[
\begin{split}
&\qquad[k\afE_{i,j}+\diag(\la-k\afbse_j)]_1\cdot[A]_1\\
&=\sum_{\dt\in\La(\infty,k)\atop a_{j,t}-\dt_t+\dt_{\bar i,\bar j}\dt_{t+i-j}\geq 0,\,\forall t\not=j}\prod_{t\in\mbz}\left({a_{i,t}+\dt_t-\dt_{\bar i,\bar j}\dt_{t+j-i}\atop\dt_t}\right)
\bigg[A+\sum_{t\in\mbz}\dt_t(\afE_{i,t}-\afE_{j,t})\bigg]_1.
\end{split}
\]

$(2)$ For $A\in\aftiThn$, we have the following triangular relation in $\afKnmbz:$
$$\prod_{(i,j)\in\sL}
[a_{i,j}\afE_{i,j}+\diag(\la^{(i,j)})]_1=[A]_1+\sum_{B\in\aftiThn\atop\sg(B^++B^-)
<\sg(A^++A^-)}h_{B,A}[B]_1\quad(h_{B,A}\in\mbz),
$$
where the products are taken with respect to any fixed total order on $\sL$ and $\la^{(i,j)}\in\afmbzn$ are uniquely determined by $A$ and the fixed total order on $\sL$. In particular,
the algebra $\afKnmbz$ is generated by the elements $[k\afE_{i,j}+\diag(\la)]_1$ with $k\in\mbn$, $1\leq i\leq n$, $j\in\mbz$, $i\not=j$ and $\la\in\afmbzn$.
\end{Lem}

Let $\afKnk=\afKnmbz\ot\field$.
We shall denote the images of $[A]_1$ in $\afKnk$ by the same letters.
Let $\msKnh$ be the $\field$-submodule of $\afKnk$ spanned by the elements $[A]_1$ with $A\in\aftiThnh$, where
$\aftiThnh=\{A=(a_{i,j})\in\aftiThn\mid a_{i,j}<p^{h},\,\forall i\not=j\}$.
The following result is the affine analogue of \cite[6.2]{BLM}.

\begin{Prop}\label{the property of K}
$(1)$ $\msKnh$ is a subalgebra of $\afKnk$. It is generated by $[kE_{i,j}+\diag(\la)]_1$ with $0\leq k<p^{h}$, $1\leq i\leq n$, $j\in\mbz$, $i\not=j$ and $\la\in\afmbzn$.

$(2)$ Let $\la\in\afmbzn$. The map
$\tau_{\la}:\msKnh\ra\msKnh$ given by $[A]_1\ra[A+p^{h}\diag(\la)]_1$ is an algebra homomorphism.
\end{Prop}
\begin{proof}
Let $A\in\aftiThnh$, $0\leq k<p^{h}$, $i\not=j$ and $\la=\ro(A)$. Let
$B=k\afE_{i,j}+\diag(\la-k\afbse_j)$.
Assume that
$A+\sum_{t\in\mbz}\dt_t(E_{i,t}-E_{j,t})\in\aftiThn\backslash\aftiThnh$ for some $\dt\in\La(\infty,k)$. Then $a_{i,t}+\dt_t-\dt_{\bar i,\bar j}\dt_{t+j-i}\geq
p^{h}$ for some $t\neq i$. Since $A\in\aftiThnh$, $\dt\in\La(\infty,k)$ and $k<p^h$,
we have $a_{i,t}+\dt_t-\dt_{\bar i,\bar j}\dt_{t+j-i}-p^h\leq
a_{i,t}+\dt_t-p^h<\dt_t$ and $\dt_t\leq k<p^h$. It follows from \ref{identity}
that
$$\bigg({a_{i,t}+\dt_t-\dt_{\bar i,\bar j}\dt_{t+j-i}\atop\dt_t}\bigg)
=\bigg({a_{i,t}+\dt_t-\dt_{\bar i,\bar j}\dt_{t+j-i}-p^h\atop\dt_t}\bigg)
=0$$
in $\field$.
Thus by \ref{formula in Knmbz}(1) we have
\begin{equation}\label{eq in the property of K}
[B]_1\cdot [A]_1\in\msKnh.
\end{equation}

We will show by induction on $\sg(A^++A^-)$ that $[A]\cdot[A']\in\msKnh$ for any $A'\in\aftiThnh$. The case $\sg(A^++A^-)=0$ is obvious. Assume now that
$\sg(A^++A^-)>0$ and that our statement is already known for  $C\in\aftiThnh$ with $\sg(C^++C^-)<\sg(A^++A^-)$.
By \ref{formula in Knmbz}(2) and \eqref{eq in the property of K} we have
$$\prod_{(i,j)\in\sL}
[a_{i,j}\afE_{i,j}+\diag(\la^{(i,j)})]_1=[A]_1+f.
$$
where $f$ is a $\field$-linear combination of $[C]_1$ with $C\in\aftiThnh$ and $\sg(C^++C^-)<\sg(A^++A^-)$. By the induction hypothesis we have $f[A']_1\in\msKnh$. Furthermore, by \eqref{eq in the property of K} we have $\prod_{(i,j)\in\sL}
[a_{i,j}\afE_{i,j}+\diag(\la^{(i,j)})]_1\cdot[A']_1\in\msKnh$. It follows that $[A]_1\cdot [A']_1\in\msKnh$. Thus (1) is proved.

From \ref{formula in Knmbz}(1) we see that $\tau_{\la}([A'']_1\cdot [A]_1)=\tau_{\la}([A'']_1)
\tau_{\la}([A]_1)$ for any $A''$ of the form $B$ as above. Since
$\msKnh$ is generated by elements like $[B]_1$ above, we conclude that
$\tau_\la$ is an algebra homomorphism.
\end{proof}

Let $\aftiThnhq$ be the set of all matrices
$A=(a_{i,j})_{i,j\in\mbz}$ with $a_{i,j}\in\mathbb N,\ a_{i,j}<p^{h}$ ($i\not= j$) and $a_{i,i}\in\mbz/p^{h}\mbz$ such that
\begin{itemize}
\item[(a)]$a_{i,j}=a_{i+n,j+n}$ for $i,j\in\mbz$; \item[(b)] for
every $i\in\mbz$, both sets $\{j\in\mbz\mid a_{i,j}\not=0\}$ and
$\{j\in\mbz\mid a_{j,i}\not=0\}$ are finite.
\end{itemize}
Let $$\afmbznh=\{(\ol{\la_i})_{i\in\mbz}\mid\ol{\la_i}\in\mbz/p^h\mbz,\,
\ol{\la_i}=\ol{\la_{i+n}},\,\forall i\}.$$
Let $\bar\ :\afmbzn\ra\afmbznh$  be the map defined by
$\ol{(\la_i)_{i}}=(\ol{\la_i})_{i}.$
There is a natural map $$pr:\aftiThnh\ra\aftiThnhq$$ sending $A+\diag(\la)$ to $A+\diag(\bar\la)$ for $A\in\afThnpmh$ and $\la\in\afmbzn$.
For $A\in\afThnpm$ and $\bar\la\in\afmbznh$, let
$\ro(A+\diag(\bar\la))=\bigl(\sum_{j\not=i}\ol{a_{i,j}}+\ol{\la_i}
\bigr)_{i}\in\afmbznh$ and
$\co(A+\diag(\bar\la))=\bigl(\sum_{i\not=j}\ol{a_{i,j}}+\la_j\bigr)_{j}\in\afmbznh$.

Let $\msKnhq$ be the free $\field$-module with basis $\{[A]_1\mid A\in\aftiThnhq\}$.
We shall define an
algebra structure on $\msKnhq$ as follows. Let $A,A'\in\aftiThnhq$. If
$\co(A)\not=\ro(A')$ in $\afmbznh$, then we define $[A]_1\cdot[A']_1$ to be zero.
Assume now that $\co(A)=\ro(A')$ in $\afmbznh$. Then there exists
$\ti A,\ti A'\in\aftiThnh$ such that $pr(\ti A)=A$, $pr(\ti A')=A'$ and $\co(\ti A)=\ro(\ti A')$ in $\afmbzn$. By
\ref{the property of K}(1), we may write $$[\ti A]_1\cdot[\ti
A']_1=\sum_{\ti A''\in \aftiThnh}\rho_{{\ti A''}}[\ti A'']_1$$ (product in
$\msKnh$), where $\rho_{{\ti A''}}\in
\field$. We define $[A]_1\cdot[A']_1$  to be $\sum_{\ti
A''\in \aftiThnh}\rho_{{\ti A''}}[pr(\ti A'')]_1$. Using \ref{the property of K}(2), one can easily prove that the product is well defined. It is easy to check that the product is associative. Consequently, $\msKnhq$ becomes an associative algebra over $\field$.

We are now ready to prove the main result of this paper.

\begin{Thm}\label{vi}
There is a $\field$-algebra isomorphism $\vih:\Unkh\ra\msKnhq$
satisfying $$
A\{\la\}\mapsto\sum_{\mu\in\afmbnnh}\bigg({\mu\atop\la}\bigg)[A+\diag(\bar\mu)]_1$$
for $A\in\afThnpmh$ and $\la\in\afmbnnh$.
\end{Thm}
\begin{proof}
Since, by \ref{basis of Unkh}, the set $\{A\{\la\}\mid A\in\afThnpmh,\,\la\in\afmbnnh\}$
forms a $\field$-basis for $\Unkh$, we see that there is a $\field$-linear map $\vih:\Unkh\ra\msKnhq$ such that
$\vih(A\{\la\})=\sum_{\mu\in\afmbnnh}\big({\mu\atop\la}\big)
[A+\diag(\bar\mu)]_1$ for $A\in\afThnpmh$ and $\la\in\afmbnnh$.
For $\la,\mu\in\afmbnnh$ we have $[\diag(\bar\la)]_1[\diag(\bar\mu)]_1
=\dt_{\la,\mu}[\diag(\bar\la)]_1$.
This together with \ref{0mu Ala} implies that
$$
\vih\bigg(\bigg({H_i\atop k}\bigg)A\{\la\}\bigg)
=\vih\bigg(\bigg({H_i\atop k}\bigg)\bigg)\vih(A\{\la\})
$$
for $1\leq i\leq n$, $0\leq k<p^h$, $A\in\afThnpmh$ and $\la\in\afmbnnh$.
Using \eqref{th+uA}, \ref{fundamental formula} and \ref{formula in Knmbz}(1) one can show that
\begin{equation*}\label{vi(kEij0Ala)}
\vih(u_{k\afE_{i,j},1}^+ A\{\la\})=\vih(u_{k\afE_{i,j},1}^+)\vih(A\{\la\})\text{ and }
\vih(u_{k\afE_{i,j},1}^- A\{\la\})=\vih(u_{k\afE_{i,j},1}^-)\vih(A\{\la\})
\end{equation*}
for $A\in\afThnpmh$, $\la\in\afmbnnh$, $i<j$ and $0\leq k<p^h$.
Consequently, $\vih$ is an algebra homomorphism,
since $\Unkh$ is generated by $\big({H_i\atop k}\big)$,
$u_{k\afE_{i,j},1}^+$ and $u_{k\afE_{i,j},1}^-$
for $0\leq k<p^h$, $1\leq i\leq n$ and $i<j$. In addition, for $A\in\afThnpm$ and $\la\in\afmbnnh$ we have
$$\vih(A\{\la\})=[A+\diag(\bar\la)]_1+\sum_{\mu\in\afmbnnh,\,\la<\mu}
\bigg({\mu\atop\la}\bigg)[A+\diag(\bar\mu)]_1.$$
It follows that the set $\{\vih(A\{\la\})\mid A\in\afThnpmh,\,\la\in\afmbnnh\}$
forms a $\field$-basis for $\msKnhq$. The theorem is proved.
\end{proof}

Let $\afhKnk$ be the $\field$-module of
all formal $\field$-linear combinations
$\sum_{A\in\aftiThn}\beta_A[A]_1$ satisfying the following property:
for any ${\bf x}\in\afmbzn$,
the sets
$\{A\in\aftiThn\ |\ \beta_A\neq0,\ \ro(A)={\bf
x}\}$ and $\{A\in\aftiThn\ |\ \beta_A\neq0,\ \co(A)={\bf
x}\}$ are finite.
We can define the product of two elements
$\sum_{A\in\aftiThn}\beta_A[A]_1$,
$\sum_{B\in\aftiThn}\gamma_B[B]_1$ in $\afhKnk$  to be
$\sum_{A,B}\beta_A\gamma_B[A]_1\cdot[B]_1$ where $[A]_1\cdot[B]_1$ is the
product in $\afKnk$. This defines an associative algebra structure on
$\afhKnk$. Note that $\sum_{\la\in\afmbzn}[\diag(\la)]_1$ is the identity element of $\afhKnk$. We are ready to show that $\msKnhq$ and the hyperalgebra $\Unk$ can all be realized as subalgebras of $\afhKnk$.

\begin{Thm}\label{psi}
There is a injective algebra homomorphism
$\psih:\msKnhq\ra\afhKnk$ such that
$\psih([A+\diag(\bar\la)]_1)=[\![A+\diag(\bar\la)]\!]_1$
for $A\in\afThnpmh$ and $\bar\la\in\afmbznh$,
where $[\![A+\diag(\bar\la)]\!]_1=\sum_{\nu\in\afmbzn,\,\bar\nu
=\bar\la}[A+\diag(\nu)]_1\in\afhKnk.$
\end{Thm}
\begin{proof}
Since the set $\{[\![A+\diag(\bar\la)]\!]_1\mid A\in\afThnpmh,\,\bar\la\in\afmbznh\}$ is linear independent, the linear map $\psih$ is injective.
It remains to show that $\psih$ is an algebra homomorphism. Let $\bar\la,\
\bar\mu\in\afmbznh,\ A, B\in\afThnpmh$. Assume that $\ol{\co(B)}+\bar\mu
=\ol{\ro(A)}+\bar\la$. Then there exist $\al,\bt\in\afmbzn$ such that $\bar\al=\bar\la$, $\bar\bt=\bar\mu$ and
$\co(B)+\bt=\ro(A)+\al$.
By definition we have in $\afhKnk$
\begin{equation*}
\begin{split}
\psih([B+\diag(\bar\mu)]_1))\cdot\psih([A+\diag(\bar\la)]_1)
&=\sum_{{\mu',\la'\in\afmbzn,\ol{\mu'}=\bar\mu,\,\ol{\la'}=\bar\la}\atop
{\co(B)+\mu'=\ro(A)+\la'}}[B+\diag(\mu')]_1
\cdot[A+\diag(\la')]_1\\
&=\sum_{{\mu',\la'\in\afmbzn,\ol{\mu'}=\bar\bt,\,\ol{\la'}=\bar\al}\atop
{\la'-\mu'=\al-\bt}}[B+\diag(\mu')]_1
\cdot[A+\diag(\la')]_1.
\end{split}
\end{equation*}
If $\la',\mu'\in\afmbzn$ is such that $\ol{\la'}=\bar\al$,
$\ol{\mu'}=\bar\bt$ and $\la'-\mu'=\al-\bt$, then
$\la'=\al+p^h\dt$ and $\mu'=\bt+p^h\dt$ for some $\dt\in\afmbzn$.
It follows that
$\psih([B+\diag(\bar\mu)]_1))\cdot\psih([A+\diag(\bar\la)]_1)
=\sum_{\dt\in\afmbzn}[B+\diag(\bt+p^h\dt)]_1\cdot[A+\diag(\al+p^h\dt)]_1.$
By \ref{the property of K}(1), we may write
$[B+\diag(\bt)]_1\cdot[A+\diag(\al)]_1=
\sum_{C\in\afThnpmh,\,\ga\in\afmbzn}\rho_{C,\ga}[C+\diag(\ga)]_1$
in $\msKnh$, where $\rho_{C,\ga}\in\field$.
Consequently, by \ref{the property of K}(2) we have
\[
\begin{split}
\psih([B+\diag(\bar\mu)]_1))\cdot\psih([A+\diag(\bar\la)]_1)
&=\sum_{\dt,\ga\in\afmbzn,\,C\in\afThnpmh}\rho_{C,\ga}[C+\diag(\ga+p^h\dt)]_1\\
&=\sum_{\ga\in\afmbzn,\,C\in\afThnpmh}\rho_{C,\ga}[\![C+\diag(\bar\ga)]\!]_1\\
&=\psih([B+\diag(\bar\mu)]_1\cdot[A+\diag(\bar\la)]_1).
\end{split}
\]
The theorem is proved.
\end{proof}

\begin{Thm}
There is an injective algebra homomorphism $\zeta_\field:\Unk\ra\afhKnk$
satisfying $$A\{\la\}\mapsto\sum_{\mu\in\afmbzn}
\bigg({\mu\atop\la}\bigg)[A+\diag(\mu)]_1$$ for $A\in\afThnpm$ and $\la\in\afmbnn$. Furthermore we have the following commutative diagram$:$
$$\xymatrix{
\Unkh \ar[rr]^{\vih}\ar[d]_{\zeta_\field|_{\Unkh}}  &  & \msKnhq \ar [dll]^{\psih} \\
\afhKnk & &
}$$
where $\vih$ is defined in \ref{vi}
and $\psih$ is defined in \ref{psi}.
\end{Thm}
\begin{proof}
From \ref{basis2 for afuglz} we see that there is a linear map $\zeta_\field:\Unk\ra\afhKnk$
such that $\zeta_\field(A\{\la\})=\sum_{\mu\in\afmbzn}
\big({\mu\atop\la}\big)[A+\diag(\mu)]_1$ for $A\in\afThnpm$ and $\la\in\afmbnn$. Clearly we have $\zeta_\field|_{\Unkh}=\psih\circ\vih$ for $h\geq 1$. This together with \ref{vi} and \ref{psi} implies that $\zeta_\field|_{\Unkh}$  is an injective algebra homomorphism for $h\geq 1$. Consequently, by \eqref{bin Unkh}, we conclude that $\zeta_\field$ is an injective algebra homomorphism.
\end{proof}

\section{The affine analogue of little and  infinitesimal Schur algebras}

Recall the map $\xrk:\Unk\tra\afSrk$ defined in \eqref{xrk}.
Let $\Unkhr=\xrk(\Unkh)$. The algebra $\Unkhr$ is the affine analogue of little Schur algebras introduced in \cite{DFW05,Fu07}.

Recall that $\afLanr=\{\la\in\afmbnn\mid\sum_{1\leq i\leq n}\la_i=r\}$. Let $\ol{\afLa(n,r)}_{p^{h}}=\{\bar\la\in\afmbznh\mid\la\in
\afLanr\}$.
For $A\in\afThnpmh$ and $\bar\la\in\afmbznh$
we define the element $[\![A+\diag(\bar\la),r]\!]_1\in\afSrk$ as follows:
\begin{equation*}\label{[[A,r]]}
[\![A+\diag(\bar\la),r]\!]_1=
\begin{cases}
\sum\limits_{\mu\in\afLa(n,r-\sg(A)) \atop
\bar\mu=\bar\la}[A+\diag(\mu)]_1 &\text{if $\sg(A)\leq r$ and $\bar\la\in\ol{\afLa(n,r-\sg(A))}_{p^{h}}$,}\\
0&\text{otherwise}.
\end{cases}
\end{equation*}
For $A\in\afThnpm$ let $\eap=\xrk(\Eap)$, $\faf=\xrk(\Faf)$
and $\bssg(A)=(\sg_i(A))_{i\in\mbz}\in\afmbnn$
where
$\sg_{i}(A)=\sum_{j<i}(a_{i,j}+a_{j,i}).$

By \ref{tri1} we have the following result.
\begin{Lem}\label{tri Schur}
For $A\in\afThnpm$ and $\la\in\afLanr$ we have the following triangular relation in $\afSrk:$
$$\eap [\diag(\la)]_1\faf=[A+\diag(\la-\bssg(A))]_1+f$$
where $f$ is a $\field$-linear combination of $[B+\diag(\mu)]_1$ with
$B\in\afThnpm$, $\mu\in\afLa(n,r-\sg(B))$ and $\sg(B)<\sg(A)$. In particular, the set $\{\eap[\diag(\la)]_1\faf\mid
A\in\afThnpm,\,\la\in\afLanr,\,\la\geq\bssg(A)\}$ forms a $\field$-basis for $\afSrk$.
\end{Lem}

\begin{Prop}
Each of the following set forms a $\field$-basis for $\Unkhr:$
\begin{itemize}
\item[(1)]
$\fB_{h,r}=\{A\{\la,r\}\mid A\in\afThnpmh,\,\la\in\afmbnnh,\,\bar\la\in\ol{\afLa(n,r-\sg(A))}\};$
\item[(2)]
$\fP_{h,r}=\{[\![A+\diag(\bar\la),r]\!]_1\mid A\in\afThnpmh,\,\bar\la\in\ol{\afLa(n,r-\sg(A))}\};$
\item[(3)]
$\fM_{h,r}=\{\eap[\![\diag(\bar\la),r]\!]_1\faf\mid A\in\afThnpmh,\,\la\geq\bssg(A),\,\la\in\afLanr\}$.
\end{itemize}
\end{Prop}
\begin{proof}
By \ref{identity}, we have
\[
\xrk(A\{\la\})=A\{\la,r\}=\sum_{\bar\mu\in\ol{\afLa(n,r-\sg(A))}}
\bigg({\mu\atop\la}\bigg)[\![A+\diag(\bar\mu),r]\!]_1.
\]
for $A\in\afThnpmh$ and $\la\in\afmbnnh$.
This together with  \ref{basis of Unkh} implies that $$\Unkhr=\spann_\field\{A\{\la,r\})\mid A\in\afThnpmh,\,\la\in\afmbnnh\}\han\spann_\field\fP_{h,r}.$$
For $A\in\afThnpmh$ let $\sX_A=\{\la\in\afmbnnh\mid\bar\la\in\ol{\afLa(n,r-\sg(A))}\}$.
Then for $A\in\afThnpmh$ and $\la\in\sX_A$, we have
\[
A\{\la,r\}=[\![A+\diag(\bar\la),r]\!]_1+
\sum_{\mu\in\sX_A,\,\mu>\la}\bigg({\mu\atop\la}\bigg)[\![A+\diag(\bar\mu),r]\!]_1.
\]
It follows that $\spann_\field\fP_{h,r}=\spann_\field\fB_{h,r}\han\Unkhr$. So each of  the sets $\fP_{h,r}$, $\fB_{h,r}$ is a $\field$-basis for $\Unkhr$.

In addition, by \ref{tri Schur} we have
$$\eap [\![\diag(\bar\la),r]\!]_1\faf=[\![A+\diag(\ol{\la-\bssg(A)}),r]\!]_1+f$$
for $A\in\afThnpmh$ and $\bar\la\in\ol{\afLanr}$,
where $f$ is a $\field$-linear combination of $[B+\diag(\mu)]_1$ with
$B\in\afThnpm$, $\mu\in\afLa(n,r-\sg(B))$ and $\sg(B)<\sg(A)$.
Since  $\Unkhr=\spann_\field\fP_{h,r}$ we conclude that $f$ must be a $\field$-linear combination of $[\![B+\diag(\bar\nu)]\!]_1$
$B\in\afThnpmh$, $\bar\nu\in\ol{\afLa(n,r-\sg(B))}$ and $\sg(B)<\sg(A)$.
Consequently, the set $\fM_{h,r}$ forms a $\field$-basis for $\Unkhr$.
\end{proof}

For $h\geq 1$ let $\snkh$ be the $\field$-subalgebra of $\Unk$ generated by $\Unkh$ and $\big({H_i\atop t}\big)$ for $1\leq i\leq n$ and $t\in\mbn$.

\begin{Lem}\label{snkh}
Each of the following set forms a $\field$-basis for $\snkh:$
\begin{itemize}
\item[(1)]
$\{\Eap\big({H\atop\la}\big)
\Faf\mid A\in\afThnpmh,\,\la\in\afmbnn\};$
\item[(2)]
$\{A\{\la\}\mid A\in\afThnpmh,\,\la\in\afmbnn\}.$
\end{itemize}
\end{Lem}
\begin{proof}
The result can be proved in a way similar to the proof of \ref{basis of Unkh}.
\end{proof}

For $h\geq 1$ and $r\in\mbn$, let $\snkhr=\xrk(\snkh)$. The algebra $\snkhr$ is the affine analogue of infinitesimal Schur algebras introduced in \cite{DNP96}.

\begin{Prop}
Each of the following set forms a $\field$-basis for $\snkhr:$
\begin{itemize}
\item[(1)]
$\fP_{h,r}'=\{[A+\diag(\la)]_1\mid A\in\afThnpmh,\,\la\in{\afLa(n,r-\sg(A))}\};$
\item[(2)]
$\fM_{h,r}'=\{\eap[\diag(\la)]_1\faf\mid A\in\afThnpmh,\,\la\geq\bssg(A),\,\la\in\afLanr\}$.
\end{itemize}
\end{Prop}
\begin{proof}
From \ref{snkh} we see that $\snkhr=\spann_\field\{A\{\la,r\}\mid A\in\afThnpmh,\,\la\in\afmbnn\}\han\spann_\field\fP_{h,r}'$.
Furthermore, for $A\in\afThnpmh$ and $\la\in\afLa(n,r-\sg(A))$ we have
$[A+\diag(\la)]_1=A\{\la,r\}\in\snkhr$. Consequently, $\fP_{h,r}'$ forms a $\field$-basis for $\snkhr$. Now using \ref{tri Schur}, we conclude that $\fM_{h,r}'$ is also a $\field$-basis for $\snkhr$.
\end{proof}

\begin{Rem}
Let $\tts(n,r)_h$ be the $\field$-submodule of $\snkhr$ spanned by the elements $[A+\diag(\la)]_1$ for $A\in\Th^\pm(n)_h$ and $\la\in\afLa(n,r-\sg(A))$, where
$\Th^\pm(n)_h$ is the subset of $\afThnpmh$ consisting of all  $A\in\afThnpmh$ such that $a_{i,j}=0$ for $1\leq i\leq n$ and  $j\not\in[1,n]$. Then $\tts(n,r)_h$ is the infinitesimal Schur algebras introduced in \cite{DNP96}.
\end{Rem}

\end{document}